\documentclass{qtds}
\usepackage{amssymb}
\usepackage{amsmath}
\usepackage{latexsym}
\usepackage{graphicx}
\usepackage{epsfig}

\newtheorem{theo}{Theorem}
\newtheorem{prop}{Proposition}

\def\be{\begin{equation}}
\def\ee{\end{equation}}
\def\bq{\begin{eqnarray}}
\def\eq{\end{eqnarray}}
\def\beq{\begin{eqnarray*}}
\def\eeq{\end{eqnarray*}}

\begin{document}

%





\authorrunninghead{J.M. Ginoux and J. Llibre}
\titlerunninghead{Canards Existence in Memristor's Circuits}

\title{Canards Existence in Memristor's Circuits}



\vspace{24pt}
{\begin{minipage}{24pc}
\footnotesize{\ \ \ \ The aim of this work is to propose an alternative method for determining the condition of existence of ``canard solutions'' for three and four-dimensional singularly perturbed systems with only one \textit{fast} variable in the \textit{folded saddle} case. This method enables to state a unique generic condition for the existence of ``canard solutions'' for such three and four-dimensional singularly perturbed systems which is based on the stability of \textit{folded singularities} of the \textit{normalized slow dynamics} deduced from a well-known property of linear algebra. This unique generic condition is perfectly identical to that provided in previous works. Application of this method to the famous three and four-dimensional memristor canonical Chua's circuits for which the classical piecewise-linear characteristic curve has been replaced by a smooth cubic nonlinear function according to the \textit{least squares method} enables to show the existence of ``canard solutions'' in such Memristor Based Chaotic Circuits.}
\end{minipage}}
\vspace{10pt}


\begin{article}

\section{Introduction}

As recalled by Fruchard and Sch\"{a}fke \cite[p. 435]{Fruchard2007}: ``In the late 1970s, under the leadership of George Reeb, a group of young researchers, Jean-Louis Callot, Francine and Marc Diener, Albert Troesch, Emile Urlacher and then, Eric Beno\^{i}t and Imme van den Berg, based in Strasbourg some in Oran and others in Tlemcen were given as research program to develop methods for ``non-standard analysis\footnote{For more details see Robinson \cite{Robinson1966} and Nelson \cite{Nelson1977}.}'' for the study of singular perturbation problems. Georges Reeb had proposed to introduce a particular control parameter \textit{a} in the original van der Pol's equation \cite{VdP1926}.

\[
\hfill \varepsilon \ddot{x} + ( 1 - x^2 ) \dot{x} + x = a \hfill
\]

The study of this equation has led this group to discover surprising solutions, which they named ``ducks\footnote{Canards in French.}''. Van der Pol relaxation oscillator is considered as the paradigm of \textit{slow}-\textit{fast} systems, \textit{i.e.} two-dimensional singularly perturbed system with one \textit{slow}  variable and one \textit{fast}. It is well-known that for the control parameter value $a = 1$, a Hopf bifurcation takes place in this system\footnote{See Callot \textit{et al.} \cite{CallotDiener1978}, Beno\^{i}t \textit{et al.} \cite{BenCalDien}, Beno\^{i}t \textit{et al.} \cite{Benoit1981a}, Beno\^{i}t \cite{Benoit1981b} and Ginoux \textit{et al.} \cite{GinouxLLibre2011}.}. So, as expected by this group of researchers, by setting $\varepsilon$ constant, one would make the amplitude of the periodic solution (limit cycle) change very fast for values of $a$ near (just below) 1. But, the results exceeded their expectations: at one critical value $a = 0.9987404512$ a very small change of this parameter's value produced an amplitude drop of about 80 \%.

According to Marc Diener \cite[p. 38]{Diener1984}: ``It was as if the existence of medium size solutions would be a ``canard\footnote{Canard = false report, from the old-French ``vendre un canard moiti\'{e}'' (Sell the half of duck). }''! Canard is now the name of a type of solution of a slow-fast differential system, to which the above ``missing'' medium-size solutions belong, that had previously been ignored.''. Another interpretation of the denomination ``canard'' also given by Diener \cite[p. 45]{Diener1984} in the the same article, is that the periodic solution resembles, for this critical parameter value, to a duck (See Fig. 1.).

\begin{figure}[htbp]
\centerline{\includegraphics[width=8.5cm,height=4.78cm]{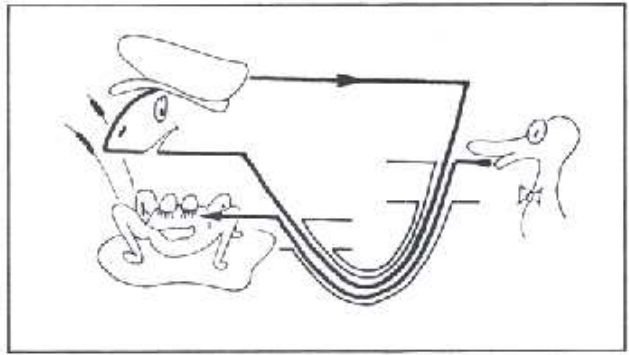}}
\vspace{0.1in}
\caption{``Canard cycle'' of the Van der Pol equation, Diener \cite[p. 45]{Diener1984}.}
\label{Fig1}
\end{figure}

In the beginning of the eighties, Beno\^{i}t and Lobry \cite{BenoitLobry}, Beno\^{i}t \cite{Benoit1983} and then Beno\^{i}t \cite{Benoit1984} in his PhD-thesis studied canard solutions in $\mathbb{R}^3$. In the article entitled ``Syst\`{e}mes lents-rapides dans $\mathbb{R}^3$ et leurs canards,''  Beno\^{i}t \cite[p. 170]{Benoit1983} proved the existence of canards solution for three-dimensional singularly perturbed systems with two \textit{slow} variables and one \textit{fast} variable while using ``Non-Standard Analysis''according to a theorem which stated that canard solutions exist in such systems provided that the \textit{pseudo singular point}\footnote{This concept has been originally introduced by Jos\'{e} Arg\'{e}mi \cite{Argemi}. See Sec. 2.7.} of the \textit{slow dynamics}, \textit{i.e.}, of the \textit{reduced vector field} is of \textit{saddle} type.

Nearly twenty years later, Szmolyan and Wechselberger \cite{SzmolyanWechselberger2001} extended ``Geometric Singular Perturbation Theory\footnote{See Fenichel \cite{Fen1971, Fen1974, Fen1977, Fen1979}, O'Malley \cite{OMalley1974}, Jones \cite{Jones1994} and Kaper \cite{Kaper1999}.}'' to canards problems in $\mathbb{R}^3$ and provided a ``standard version'' of Beno\^{i}t's theorem \cite{Benoit1983}. Very recently, Wechselberger \cite{Wechselberger2012} generalized this theorem for $n$-dimensional singularly perturbed systems with $k$ \textit{slow} variables and $m$ \textit{fast} (Eq. (\ref{eq1})). The method used by Szmolyan and Wechselberger \cite{SzmolyanWechselberger2001} and Wechselberger \cite{Wechselberger2012} require to implement a ``desingularization procedure'' which can be summarized as follows: first, they compute the \textit{normal form} of such singularly perturbed systems (see Eq. (\ref{eq28}) for dimension three and Eq. (\ref{eq48}) for dimension four) which is expressed according to some coefficients ($a$ and $b$ for dimension three and $\tilde{a}$, $\tilde{b}$ and $\tilde{c}_j$ for dimension four) depending on the functions defining the original vector field (\ref{eq1}) and their partial derivatives with respect to the variables. Secondly, they project the ``desingularized vector field'' (originally called ``normalized slow dynamics'' by Eric Beno\^{i}t \cite[p. 166]{Benoit1983}) of such a \textit{normal form} on the tangent bundle of the critical manifold. Finally, they evaluate the Jacobian of the projection of this ``desingularized vector field'' at the \textit{folded singularity} (originally called \textit{pseudo singular points} by Jos\'{e} Arg\'{e}mi \cite[p. 336]{Argemi}. This leads Szmolyan and Wechselberger \cite[p. 427]{SzmolyanWechselberger2001} and Wechselberger \cite[p. 3298]{Wechselberger2012} to a ``classification of \textit{folded singularities} (\textit{pseudo singular points})''. Thus, they show that for three-dimensional singularly perturbed systems such \textit{folded singularities} is of \textit{saddle type} if the following condition is satisfied: $a<0$ while for four-dimensional singularly perturbed systems such \textit{folded singularities} is of \textit{saddle type} if $\tilde{a}<0$. Then, Szmolyan and Wechselberger \cite[p. 439]{SzmolyanWechselberger2001} and Wechselberger \cite[p. 3304]{Wechselberger2012} establish their Theorem 4.1. which state that ``\textit{In the folded saddle and in the folded node case singular canards perturb to maximal canard for sufficiently small $\varepsilon$}''. However, in their works neither Szmolyan and Wechselberger \cite{SzmolyanWechselberger2001} nor Wechselberger \cite{Wechselberger2012} do not provide (to our knowledge) the expression of these constants ($a$ and $\tilde{a}$) which are necessary to state the existence of canard solutions in such systems.

So, the aim of this work is first to provide the expression of these constants and then to show that they can be directly determined starting from the \textit{normalized slow dynamics} and not from the projection of the ``desingularized vector field'' of the \textit{normal form}. This method enables to state a unique ``generic'' condition for the existence of ``canard solutions'' for such three and four-dimensional singularly perturbed systems which is based on the stability of \textit{folded singularities} of the \textit{normalized slow dynamics} deduced from a well-known property of linear algebra. This unique condition which is completely identical to that provided by Beno\^{i}t \cite{Benoit1983} and then by Szmolyan and Wechselberger \cite{SzmolyanWechselberger2001} and finally by Wechselberger \cite{Wechselberger2012} is ``generic'' since it is exactly the same for singularly perturbed systems of dimension three and four with only one \textit{fast} variable. So, it provides a path to many applications.

In the very beginning of the seventies, Leon Chua \cite{Chua1971} considered the three basic building blocks of an electric circuit: the capacitor, the resistor and the inductor as well as the three laws linking the four fundamental circuit variables, namely, the electric \textit{current} $i$, the \textit{voltage} $v$, the \textit{charge} $q$ and the \textit{magnetic flux} $\varphi$. He thus concluded from the logical as well as axiomatic points of view, that it is necessary, for the sake of \textit{completeness}, to postulate the existence of a fourth circuit element to which he gave the name \textit{memristor} since it behaves like a nonlinear resistor with memory. On April 30$^{th}$ 2008, Stan Williams and co-workers \cite{Strukhov2008} announced in the journal \textit{Nature} that the missing circuit element, postulated thirty-seven years before by Leon Chua has been found \cite{GinouxRossetto2013}. Since, the memristor has been subject to many studies and applications \cite{DiVentraPershinChua2009,PershinDiVentra2009}. More particularly, memristor-based circuits have been used by Itoh and Chua \cite{ItohChua2008,ItohChua2013}, Muthuswamy and Kokate \cite{MuthuswamyKokate2009}, Muthuswamy \cite{Muthuswamy2010}, Muthuswamy and Chua \cite{MuthuswamyChua2010} and Fitch \textit{et al.} \cite{Fitch2012,Fitch2013} to construct dynamical systems whose solutions exhibit chaotic and hyperchaotic behavior \cite{Fitch2012}.

In a paper entitled ``Duality of Memristors Circuits'', Itoh and Chua \cite[p. 1330001-15]{ItohChua2013} gave the memristor canonical Chua's circuit equation (69) in the three-dimensional flux-linkage and charge phase space. Differentiating this Eq. (69) with respect to time they obtained memristor-based canonical Chua's circuit equation (73) in the four-dimensional current-voltage phase space\footnote{Let's notice that Eq. (73) corresponds exactly to what Itoh and Chua \cite[p. 3188]{ItohChua2008} have called in their previous paper on ``Memristor Oscillators'' the fourth-order memristor-based canonical Chua's circuit equation (35).}.
In both cases, the $\varphi - q$ characteristic curve of these circuits has been represented by a piecewise-linear function (Eq. (40) in Itoh and Chua \cite[p. 3189]{ItohChua2008} and Eq. (70) in Itoh and Chua \cite[p. 1330001-15]{ItohChua2013}). In their works, Itoh and Chua \cite{ItohChua2008,ItohChua2013} have shown that the dynamical systems modeling such circuits possess at least one eigenvalue with a large negative real part. This specific feature is of great interest since it enables to consider memristor-based canonical Chua's circuits as \textit{slow-fast dynamical systems}. So, there exists in the phase-space a \textit{slow manifold} on which trajectories, solution of the dynamical system modeling the memristor circuit, evolve slowly and toward which nearby orbits contract exponentially in time in the normal directions.

For these memristor-based canonical circuits Itoh and Chua \cite{ItohChua2008,ItohChua2013} have used a classical piecewise-linear function for the $\varphi - q$ characteristic curve. However, Muthuswamy \cite{Muthuswamy2010} and Fitch \textit{et al.} \cite{Fitch2012} have proposed to replace this piecewise linear characteristic curve by a smooth cubic nonlinear function. This enables to exhibit the existence of generic ``canard solutions'' in such Memristor Based Chaotic Circuits.

The outline of this paper is as follows. In Sec. 2, definitions of singularly perturbed system, critical manifold, reduced system, ``constrained system'', canard cycles, folded singularities and pseudo singular points are recalled. The method proposed in this article is presented in Sec. 3 \& 4 for the case of three and four-dimensional singularly perturbed systems with only one \textit{fast} variable. Existence of canard solution
for the third and fourth-order Chua's memristor is established according to this method in Sec. 5 \& 6.

\section{Definitions}

\subsection{Singularly perturbed systems}
\label{Sec2}

According to Tikhonov \cite{Tikhonov1948}, Pontryagin \cite{Pontryagin1957}, Jones \cite{Jones1994} and Kaper \cite{Kaper1999} \textit{singularly perturbed systems} are defined as:

\begin{equation}
\label{eq1}
\begin{aligned}
{\vec {x}}' & = \varepsilon \vec{f} \left( {\vec{x},\vec{y},\varepsilon} \right), \\
{\vec {y}}' & = \vec {g}\left( {\vec{x}, \vec{y},\varepsilon }
\right).
\end{aligned}
\end{equation}

where $\vec {x} \in \mathbb{R}^k$, $\vec {y} \in \mathbb{R}^m$, $\varepsilon \in \mathbb{R}^ + $, and the prime denotes
differentiation with respect to the independent variable $t'$. The functions $\vec {f}$ and $\vec {g}$ are assumed to be $C^\infty$
functions\footnote{In certain applications these functions will be supposed to be $C^r$, $r \geqslant 1$.} of $\vec {x}$, $\vec {y}$
and $\varepsilon$ in $U\times I$, where $U$ is an open subset of $\mathbb{R}^k\times \mathbb{R}^m$ and $I$ is an open interval
containing $\varepsilon = 0$.

\smallskip

In the case when $0 < \varepsilon \ll 1$, \textit{i.e.} $\varepsilon$ is a small positive number, the variable $\vec {x}$ is called \textit{slow} variable, and $\vec {y}$ is called \textit{fast} variable. Using Landau's notation: $O\left( \varepsilon^p \right)$ represents a function $f$ of $u$ and $\varepsilon $ such that $f(u,\varepsilon) / \varepsilon^p$ is bounded for positive $\varepsilon$  going to zero, uniformly for $u$ in the given domain.

\smallskip

In general we consider that $\vec {x}$ evolves at an $O\left( \varepsilon \right)$ rate; while $\vec {y}$ evolves at an $O\left( 1 \right)$ \textit{slow} rate. Reformulating system (\ref{eq1}) in terms of the rescaled variable $t = \varepsilon t'$, we obtain

\begin{equation}
\label{eq2}
\begin{aligned}
\dot {\vec {x}} & = \vec{f} \left( {\vec{x},\vec{y},\varepsilon} \right), \\
\varepsilon \dot {\vec {y}} & = \vec {g}\left( {\vec{x}, \vec{y},\varepsilon }
\right).
\end{aligned}
\end{equation}

The dot represents the derivative with respect to the new independent variable $t$.

\smallskip

The independent variables $t'$ and $t$ are referred to the \textit{fast} and \textit{slow} times, respectively, and (\ref{eq1}) and (\ref{eq2}) are called the \textit{fast} and \textit{slow} systems, respectively. These systems are equivalent whenever $\varepsilon \ne 0$, and they are labeled \textit{singular perturbation problems} when $0 < \varepsilon \ll 1$. The label ``singular'' stems in part from the discontinuous limiting behavior
in system (\ref{eq1}) as $\varepsilon \to 0$.

\smallskip

\subsection{Reduced slow system}

In such case system (\ref{eq2}) leads to a differential-algebraic system (D.A.E.) called \textit{reduced slow system} whose dimension decreases from $k + m = n$ to $m$. Then, the \textit{slow} variable $\vec {x} \in \mathbb{R}^k$ partially evolves in the submanifold $M_0$ called the \textit{critical manifold}\footnote{It represents the approximation of the slow invariant manifold, with an error of $O(\varepsilon)$.}. The \textit{reduced slow system} is

\begin{equation}
\label{eq3}
\begin{aligned}
\dot {\vec {x}} & = \vec{f} \left( {\vec{x},\vec{y},\varepsilon} \right), \\
\vec {0} & = \vec {g}\left( {\vec{x}, \vec{y},\varepsilon }
\right).
\end{aligned}
\end{equation}

\subsection{Slow Invariant Manifold}

The \textit{critical manifold} is defined by

\begin{equation}
\label{eq4} M_0 := \left\{ {\left( {\vec {x},\vec {y}} \right):\vec
{g}\left( {\vec {x},\vec {y},0} \right) = {\vec {0}}} \right\}.
\end{equation}

Such a normally hyperbolic invariant manifold (\ref{eq4}) of the \textit{reduced slow system} (\ref{eq3}) persists as a locally invariant \textit{slow manifold} of the full problem (\ref{eq1}) for $\varepsilon$ sufficiently small. This locally \textit{slow invariant manifold} is $O(\varepsilon)$ close to the \textit{critical manifold}.

When $D_{\vec{x}}\vec{f}$ is invertible, thanks to the Implicit Function Theorem, $M_0 $ is given by the graph of a $C^\infty $ function $\vec {x} = \vec {G}_0 \left( \vec {y} \right)$ for $\vec {y} \in D$, where $D\subseteq \mathbb{R}^k$ is a compact, simply connected domain and the boundary of D is a $(k - 1)$--dimensional $C^\infty$ submanifold\footnote{The set D is overflowing invariant with respect to (\ref{eq2}) when $\varepsilon = 0$. See Kaper \cite{Kaper1999} and Jones \cite{Jones1994}.}.

\smallskip

According to Fenichel \cite{Fen1971,Fen1979} theory if $0 < \varepsilon \ll 1$ is sufficiently small, then there exists a function $\vec {G}\left( {\vec {y},\varepsilon } \right)$ defined on D such that the manifold

\begin{equation}
\label{eq5} M_\varepsilon := \left\{ {\left( {\vec {x},\vec {y}}
\right):\vec {x} = \vec {G}\left( {\vec {y},\varepsilon } \right)} \right\},
\end{equation}

is locally invariant under the flow of system (\ref{eq1}). Moreover, there exist perturbed local stable (or attracting) $M_a$ and unstable (or repelling) $M_r$ branches of the \textit{slow invariant manifold} $M_\varepsilon$. Thus, normal hyperbolicity of $M_\varepsilon$ is lost via a saddle-node bifurcation of the \textit{reduced slow system} (\ref{eq3}). Then, it gives rise to solutions of ``canard'' type.

\subsection{Canards, singular canards and maximal canards}

A \textit{canard} is a solution of a singularly perturbed dynamical system (\ref{eq1}) following the \textit{attracting} branch $M_a$ of the \textit{slow invariant manifold}, passing near a bifurcation point located on the fold of this \textit{slow invariant manifold}, and then following the \textit{repelling} branch $M_r$ of the \textit{slow invariant manifold}.\\

A \textit{singular canard} is a solution of a \textit{reduced slow system} (\ref{eq3}) following the \textit{attracting} branch $M_{a,0}$ of the \textit{critical manifold}, passing near a bifurcation point located on the fold of this \textit{critical manifold}, and then following the \textit{repelling} branch $M_{r,0}$ of the \textit{critical manifold}.\\

A \textit{maximal canard} corresponds to the intersection of the attracting and repelling branches $M_{a,\varepsilon} \cap M_{r,\varepsilon}$ of the slow manifold in the vicinity of a non-hyperbolic point.\\

According to Wechselberger \cite[p. 3302]{Wechselberger2012}:

\begin{quote}
``Such a maximal canard defines a family of canards nearby which are exponentially close to the maximal canard, \textit{i.e.} a family of solutions of (\ref{eq1}) that follow an attracting branch $M_{a,\varepsilon}$ of the slow manifold and then follow, rather surprisingly, a repelling/saddle branch
$M_{r,\varepsilon}$ of the slow manifold for a considerable amount of slow time. The existence of this family of canards is a consequence of the non-uniqueness of $M_{a,\varepsilon}$ and $M_{r,\varepsilon}$. However, in the singular limit $\varepsilon \rightarrow 0$, such a family of canards is represented by a unique singular canard.''
\end{quote}

Canards are a special class of solution of singularly perturbed dynamical systems for which normal hyperbolicity is lost. Canards in singularly perturbed systems with two or more slow variables $(\vec {x} \in \mathbb{R}^k$, $k \geqslant 2)$ and one fast variable $(\vec {y} \in \mathbb{R}^m$, $m = 1)$ are robust, since maximal canards generically persist under small parameter changes\footnote{See Beno\^{i}t \cite{Benoit1983, Benoit2001}, Szmolyan and Wechselberger \cite{SzmolyanWechselberger2001} and Wechselberger \cite{Wechselberger2005,Wechselberger2012}.}.

\subsection{Constrained system}

In order to characterize the ``slow dynamics'',  \textit{i.e.} the slow trajectory of the \textit{reduced slow system} (\ref{eq3})  (obtained by setting $\varepsilon = 0$ in (\ref{eq2})), Floris Takens \cite{Takens1976} introduced the ``constrained system'' defined as follows:

\begin{equation}
\label{eq6}
\begin{aligned}
\dot {\vec {x}} & = \vec{f} \left( {\vec{x},\vec{y},0} \right), \\
D_{\vec{y}} \vec{g} . \dot {\vec {y}} & =  - ( D_{\vec{x}} \vec{g} . \vec {f} ) \left( {\vec{x}, \vec{y}, 0} \right),\\
\vec {0} & = \vec {g}\left( {\vec{x}, \vec{y}, 0 } \right).
\end{aligned}
\end{equation}

\smallskip

Since, according to Fenichel \cite{Fen1971,Fen1979}, the \textit{critical manifold} $\vec{g} \left( {\vec{x},\vec{y},0} \right)$ may be considered as locally invariant under the flow of system (\ref{eq1}), we have:

\[
\hfill \frac{d\vec{g}}{dt} \left( {\vec{x},\vec{y},0} \right) = 0 \quad \Longleftrightarrow \quad D_{\vec{x}} \vec{g} . \dot{\vec{x}} + D_{\vec{y}} \vec{g} . \dot {\vec {y}} = \vec{0}. \hfill
\]

By replacing $\dot{\vec{x}}$ by $\vec{f}\left( {\vec{x},\vec{y},0} \right)$ leads to:

\[
\hfill D_{\vec{x}} \vec{g} . \vec{f}\left( {\vec{x},\vec{y},0} \right) + D_{\vec{y}} \vec{g} . \dot {\vec {y}} = \vec{0}. \hfill
\]

This justifies the introduction of the \textit{constrained system}.\\

Now, let $adj(D_{\vec{y}} \vec{g})$ denote the adjoint of the matrix $D_{\vec{y}} \vec{g}$ which is the transpose of the co-factor matrix $D_{\vec{y}} \vec{g}$, then while multiplying the left hand side of (\ref{eq6}) by the inverse matrix $(D_{\vec{y}} \vec{g})^{-1}$ obtained by the adjoint method we have:

\begin{equation}
\label{eq7}
\begin{aligned}
\dot {\vec {x}} & = \vec{f} \left( {\vec{x},\vec{y},0} \right), \\
det(D_{\vec{y}} \vec{g}) \dot {\vec {y}} & =  - (adj(D_{\vec{y}} \vec{g}) . D_{\vec{x}} \vec{g} . \vec {f} ) \left( {\vec{x}, \vec{y}, 0} \right),\\
\vec {0} & = \vec {g}\left( {\vec{x}, \vec{y}, 0 } \right).
\end{aligned}
\end{equation}

\subsection{Normalized slow dynamics}

Then, by rescaling the time by setting $t = - det(D_{\vec{y}} \vec{g}) \tau$ we obtain the following system which has been called by Eric Beno\^{i}t \cite[p. 166]{Benoit1983} ``normalized slow dynamics'':

\begin{equation}
\label{eq8}
\begin{aligned}
\dot {\vec {x}} & = - det(D_{\vec{y}} \vec{g}) \vec{f} \left( {\vec{x},\vec{y},0} \right), \\
\dot {\vec {y}} & =  (adj(D_{\vec{y}} \vec{g}) . D_{\vec{x}} \vec{g} . \vec {f} ) \left( {\vec{x}, \vec{y}, 0} \right),\\
\vec {0} & = \vec {g}\left( {\vec{x}, \vec{y}, 0 } \right).
\end{aligned}
\end{equation}

where the overdot now denotes the time derivation with respect to $\tau$.

Let's notice that Jos\'{e} Arg\'{e}mi \cite{Argemi} proposed to rescale time by setting $t = - det(D_{\vec{y}} \vec{g}) sgn(det(D_{\vec{y}} \vec{g})) \tau$ in order to keep the same flow direction in (\ref{eq8}) as in (\ref{eq7}).

\subsection{Desingularized vector field}

By application of the Implicit Function Theorem, let suppose that we can explicitly express from Eq. (\ref{eq4}), say without loss of generality, $x_1$ as a function $\phi_1$ of the other variables. This implies that $M_0$ is locally the graph of a function $\phi_1 \mbox{ : } \mathbb{R}^k \to \mathbb{R}^m$ over the base $U = ( \vec{ \chi }, \vec{y} ) $ where $\vec{ \chi } = (x_2, x_3, ..., x_k)$. Thus, we can span the ``normalized slow dynamics'' on the tangent bundle at the \textit{critical manifold} $M_0$ at the \textit{pseudo singular point}. This leads to the so-called \textit{desingularized vector field}:

\begin{equation}
\label{eq9}
\begin{aligned}
\dot {\vec {\chi}} & = - det(D_{\vec{y}} \vec{g}) \vec{f} \left( {\vec{\chi},\vec{y},0} \right), \\
\dot {\vec {y}} & =  (adj(D_{\vec{y}} \vec{g}) . D_{\vec{x}} \vec{g} . \vec {f} ) \left( {\vec{\chi}, \vec{y}, 0} \right).
\end{aligned}
\end{equation}

\subsection{Pseudo singular points and pseudo singular manifolds}

As recalled by Guckenheimer and Haiduc \cite[p. 91]{GuckenHaiduc2005}, \textit{pseudo-singular points} have been introduced by the late Jos\'{e} Arg\'{e}mi \cite{Argemi} for low-dimensional singularly perturbed systems and are defined as singular points of the ``normalized slow dynamics'' (\ref{eq8}). Twenty-three years later, Szmolyan and Wechselberger \cite[p. 428]{SzmolyanWechselberger2001} called such \textit{pseudo singular points}, \textit{folded singularities}. In a recent publication entitled ``A propos de canards'' Wechselberger \cite[p. 3295]{Wechselberger2012} proposed to define such singularities for $n$-dimensional singularly perturbed systems with $k$ \textit{slow} variables and $m$ \textit{fast} as the solutions of the following system:

\begin{equation}
\label{eq10}
\begin{aligned}
& det(D_{\vec{y}} \vec{g}) = 0, \\
& (adj(D_{\vec{y}} \vec{g}) . D_{\vec{x}} \vec{g} . \vec {f} ) \left( {\vec{x}, \vec{y}, 0} \right) = \vec{0},\\
& \vec {g}\left( {\vec{x}, \vec{y}, 0 } \right) = \vec {0}.
\end{aligned}
\end{equation}

Thus, for dimensions higher than three, his concept encompasses that of Arg\'{e}mi. Moreover, Wechselberger \cite[p. 3296]{Wechselberger2012} proved that \textit{folded singularities} form a $(k-2)$-dimensional manifold. Thus, for $k=2$ the \textit{folded singularities} are nothing else than the \textit{pseudo singular points} defined by Arg\'{e}mi \cite{Argemi}. While for $k \geqslant 3$ the \textit{folded singularities} are no more points but a $(k-2)$-dimensional manifold. Moreover, let's notice on the one hand that the original system (\ref{eq1}) includes $n=k+m$ variables and on the other hand, that the system (\ref{eq10}) comprises $p=2m +1$ equations. So, we are faced to a system of $p$ equations with $n$ unknowns. If $p < n$ the system is triangular and will necessarily have an infinite number of solutions that will be able to express in terms of the last unknowns. Since in this work we are only interested in three and four-dimensional singularly perturbed systems with $m=1$ \textit{fast} variable and with $k$ \textit{slow} variables we have $p=3$ and $n=k+1$ and so, $n-p=k-2$. Thus, we will examine the case $k=2$ and $k \geqslant 3$.

\subsubsection{Pseudo singular points -- Case $k = 2$}

If $k=2$ the number of variables of system (\ref{eq1}) is equal to $n=3$ and the number of equations is also equal to $p=3$. So, all the variables (unknowns) of system (\ref{eq10}) can be determined. The solutions of such system are called \textit{pseudo singular points}. An example of such situation is given by the third-order Memristor-Based canonical oscillator analyzed in Sec. 5 and for which $(m,k) = (1,2)$. We will see in the next Sec. 3 that the stability analysis of these \textit{pseudo singular points} will give rise to a condition for the existence of canard solutions in such systems.

\subsubsection{Pseudo singular manifolds -- Case $k \geqslant 3$}

If $k \geqslant 3$ the number of variables of system (\ref{eq1}) is equal to $n=k+1$ and the number of equations is still equal to $p=3$. So, only three variables (unknowns) of system (\ref{eq10}) can be determined while the remaining $k-2$ unknowns are undetermined. The solution of such system takes the form of a $(k-2)$-dimensional manifold that we call \textit{pseudo singular manifold}. An example of such situation is given by the fourth-order Memristor-Based canonical oscillator analyzed in Sec. 6 and for which $(m,k) = (1,3)$. We will see in Sec. 4 that for $k \geqslant 3$ the stability analysis of this \textit{pseudo singular manifold} will give rise to a condition (represented by a domain) for the existence of canard solutions in such systems.\\

\section{Three-dimensional singularly perturbed systems}
\label{Sec3}

A three-dimensional \textit{singularly perturbed dynamical system} (\ref{eq2}) with $k=2$ \textit{slow} variables and $m=1$ \textit{fast} may be written as:

\begin{subequations}
\label{eq11}
\begin{align}
 \dot{x}_1 & = f_1 \left( x_1, x_2, y_1 \right), \hfill \\
 \dot{x}_2 & = f_2 \left( x_1, x_2, y_1 \right), \hfill \\
 \varepsilon \dot{y}_1 & = g_1 \left( x_1, x_2, y_1 \right), \hfill
\end{align}
\end{subequations}

\smallskip

where $\vec{x}= (x_1, x_2)^t \in \mathbb{R}^2$, $\vec{y}= (y_1) \in \mathbb{R}$, $0 <~\varepsilon \ll~1$ and the functions $f_i$ and $g_i$ are assumed to be $C^2$ functions of $(x_1, x_2, y_1)$.

\subsection{Critical Manifold}

The critical manifold equation of system (\ref{eq11}) is defined by setting $\varepsilon = 0$ in Eq. (11c). Thus, we obtain:

\begin{equation}
\label{eq12}
g_1 \left( x_1, x_2, y_1 \right) = 0.
\end{equation}

\smallskip

By application of the Implicit Function Theorem, let suppose that we can explicitly express from Eq. (\ref{eq11}), say without loss of generality, $x_1$ as functions of the others variables:

\[
x_1 = \phi \left(x_2, y_1 \right)
\]

\subsection{Constrained system}

The \textit{constrained system} of system (\ref{eq11}) is obtained by equating to zero the time derivative of $g_1 \left(x_1, x_2, y_1 \right)$:

\begin{equation}
\label{eq13}
\frac{dg_1}{dt} = \frac{\partial g_1}{\partial x_1} \dot{x}_1 + \frac{\partial g_1}{\partial x_2} \dot{x}_2 + \frac{\partial g_1}{\partial y_1} \dot{y}_1 = 0
\end{equation}

By replacing $\dot{x}_i$ by $f_i \left( x_1, x_2, y_1 \right)$ with $i=1,2$, Eq. (\ref{eq13}) reads:

\begin{equation}
\label{eq14}
\dot{y}_1 = - \dfrac{ \dfrac{\partial g_1}{\partial x_1} f_1 + \dfrac{\partial g_1}{\partial x_2} f_2}{ \dfrac{\partial g_1}{\partial y_1} }.
\end{equation}

So, we have the following constrained system:

\begin{equation}
\label{eq15}
\begin{aligned}
 \dot{x}_1 & = f_1 \left( x_1, x_2, y_1 \right), \hfill \\
 \dot{x}_2 & = f_2 \left( x_1, x_2, y_1 \right), \hfill \\
 \dot{y}_1 & = - \dfrac{ \dfrac{\partial g_1}{\partial x_1} f_1 + \dfrac{\partial g_1}{\partial x_2} f_2 }{ \dfrac{\partial g_1}{\partial y_1} }, \hfill \\
 0 & = g_1 \left( x_1, x_2, y_1 \right). \hfill
\end{aligned}
\end{equation}

\subsection{Normalized slow dynamics}

By rescaling the time by setting $t = -  \dfrac{\partial g_1}{\partial y_1} \tau$ we obtain the ``normalized slow dynamics'':

\begin{equation}
\label{eq16}
\begin{aligned}
\dot{x}_1 & = - f_1 \left( x_1, x_2, y_1 \right) \dfrac{\partial g_1}{\partial y_1} = F_1 \left(x_1, x_2, y_1 \right), \hfill \vspace{6pt} \\
 \dot{x}_2 & = - f_2 \left( x_1, x_2, y_1 \right) \dfrac{\partial g_1}{\partial y_1} = F_2 \left(x_1, x_2, y_1 \right), \hfill \vspace{6pt} \\
 \dot{y}_1 & = \dfrac{\partial g_1}{\partial x_1} f_1 + \dfrac{\partial g_1}{\partial x_2} f_2  = G_1 \left(x_1, x_2, y_1 \right), \hfill \vspace{6pt} \\
 0 & = g_1 \left( x_1, x_2, y_1 \right). \hfill
\end{aligned}
\end{equation}

where the overdot now denotes the time derivation with respect to $\tau$\footnote{In the three-dimensional case $det(D_{\vec{y}} \vec{g}) = \partial g_1 / \partial y_1$.}.

\subsection{Desingularized  vector field}

Then, since we have supposed that $x_1$ may be explicitly expressed as a function $\phi \left(x_2, y_1 \right)$ on the others variables (Eq. \ref{eq12}), it can be used to project the ``normalized slow dynamics'' (\ref{eq16}) on the tangent bundle of the critical manifold. Thus we obtain the so-called ``desingularized vector field''

\begin{equation}
\label{eq17}
\begin{aligned}
 \dot{x_2} & = - f_2 \left( x_1 , x_2, y_1 \right) \frac{\partial g_1}{\partial y_1}\left( x_1, x_2, y_1 \right), \hfill \\
 \dot{y_1} & = \frac{\partial g_1}{\partial x_1} f_1\left(x_1, x_2, y_1 \right) + \frac{\partial g_1}{\partial x_2} f_2\left(x_1, x_2, y_1 \right).
\end{aligned}
\end{equation}

\smallskip

in which $x_1$ must be replaced by $\phi \left(x_2, y_1 \right)$.

\subsection{Pseudo-Singular Points}

\textit{Pseudo-singular points} are defined as singular points of the ``normalized slow dynamics'' (\ref{eq16}), \textit{i.e.} as the set of points for which we have:

\begin{subequations}
\label{eq18}
\begin{align}
 & g_1 \left( x_1, x_2, y_1 \right) = 0, \hfill \\
 & \dfrac{\partial g_1}{\partial y_1} = 0, \hfill \\
 & \dfrac{\partial g_1}{\partial x_1} f_1 + \dfrac{\partial g_1}{\partial x_2} f_2 = 0. \hfill
\end{align}
\end{subequations}

\smallskip

\begin{remark}
According to Arg\'{e}mi \cite{Argemi}, \textit{pseudo singular points} are singular points of (\ref{eq18}) but not necessarily singular points of (\ref{eq11}). In the following, we do not consider the case for which $f_1 ( x_1, x_2, y_1 ) = f_2 ( x_1, x_2, y_1 ) = g_1 ( x_1, x_2, y_1) = 0$. Let's notice that contrary to the previous works we don't use the ``desingularized vector field'' (\ref{eq17}) but the ``normalized slow dynamics'' (\ref{eq16}).
\end{remark}

Thus, the Jacobian matrix of system (\ref{eq16}) reads:

\begin{equation}
\label{eq19}
J_{(F_1, F_2, G_1)} = \begin{pmatrix}
\dfrac{\partial F_1}{\partial x_1 } \quad & \quad \dfrac{\partial F_1}{\partial x_2 } \quad & \quad \dfrac{\partial F_2}{\partial y_1 } \vspace{6pt} \\
\dfrac{\partial F_2}{\partial x_1 } \quad & \quad \dfrac{\partial F_2}{\partial x_2 } \quad &  \quad \dfrac{\partial F_2}{\partial y_1 } \vspace{6pt} \\
\dfrac{\partial G_1}{\partial x_1 } \quad & \quad \dfrac{\partial G_1}{\partial x_2 } \quad &  \quad \dfrac{\partial G_1}{\partial y_1 } \vspace{6pt}
\end{pmatrix}
\end{equation}

\smallskip

\subsection{Beno\^{i}t's generic hypothesis}

In his famous papers, Eric Beno\^{i}t \cite{BenoitLobry,Benoit1983,Benoit1990} made the following assumptions without loss of generality. First, he supposed that by a ``standard translation'' the \textit{pseudo-singular point} can be shifted at the origin and that by a ``standard rotation'' of $y_1$-axis that the \textit{critical manifold} (\ref{eq12}) is tangent to ($x_2,y_1$)-plane, so he had:

\begin{equation}
\label{eq20}
f_1 \left(0, 0, 0 \right) = g_1 \left( 0, 0, 0 \right) = \left.  \dfrac{\partial g_1}{\partial x_2} \right|_{(0,0,0)} = \left.  \dfrac{\partial g_1}{\partial y_1} \right|_{(0,0,0)} = 0.
\end{equation}

\smallskip

Then, he made the following assumptions for the non-degeneracy of the \textit{pseudo-singular point}:

\begin{equation}
\label{eq21}
f_2 \left(0, 0, 0 \right) \neq 0 \mbox{ ;}  \left.  \dfrac{\partial g_1}{\partial x_1} \right|_{(0,0,0)} \neq 0 \mbox{ ;} \left.  \dfrac{\partial^2 g_1}{\partial y_1^2} \right|_{(0,0,0)} \neq 0
\end{equation}

\smallskip

According to Beno\^{i}t's generic hypotheses Eqs. (\ref{eq20}-\ref{eq21}), the Jacobian matrix (\ref{eq19}) reads:

\begin{equation}
\label{eq22}
J_{(F_1, F_2, G_1)} = \begin{pmatrix}
0 & 0 & 0 \vspace{6pt} \\
-f_2 \dfrac{\partial^2 g_1}{\partial x_1 \partial y_1} &  -f_2 \dfrac{\partial^2 g_1}{\partial x_2 \partial y_1} &  - f_2 \dfrac{\partial^2 g_1}{\partial y_1^2} \vspace{6pt} \\
a_{31} & a_{32} & a_{33} \vspace{6pt}
\end{pmatrix}
\end{equation}

\smallskip

where

\[
\begin{aligned}
a_{3i} & = \dfrac{\partial g_1}{\partial x_1}\dfrac{\partial f_1}{\partial x_i} + f_2 \dfrac{\partial^2 g_1}{\partial x_2 \partial x_i} \mbox{ for } i =1,2,\\
a_{33} & = \dfrac{\partial g_1}{\partial x_1}\dfrac{\partial f_1}{\partial y_1} + f_2 \dfrac{\partial^2 g_1}{\partial x_2 \partial y_1}.
\end{aligned}
\]

\smallskip

Thus, we have the following Cayley-Hamilton eigenpolynomial associated with such a Jacobian matrix (\ref{eq22}) evaluated at the \textit{pseudo singular point}, \textit{i.e.}, at the origin:

\begin{equation}
\label{eq23}
\lambda^3 - \sigma_1 \lambda^2 + \sigma_2 \lambda - \sigma_3 = 0
\end{equation}

\smallskip

It appears that $\sigma_3 = |J_{(F_1, F_2, G_1)}| = 0$ since one row of the Jacobian matrix (\ref{eq22}) is null. So, the Cayley-Hamilton eigenpolynomial reduces to:

\begin{equation}
\label{eq24}
\lambda \left(\lambda^2 - \sigma_1 \lambda + \sigma_2 \right) = 0
\end{equation}

\smallskip

Let $\lambda_i$ be the eigenvalues of the eigenpolynomial (\ref{eq24}) and let's denote by $\lambda_3 = 0$ the obvious root of this polynomial. We have:

\begin{equation}
\label{eq25}
\begin{aligned}
\sigma_1 = & Tr(J_{(F_1, F_2, G_1)}) = \lambda_1 + \lambda_2 = \dfrac{\partial g_1}{\partial x_1}\dfrac{\partial f_1}{\partial y_1}, \hfill  \\
\sigma_2  = & \sum_{i=1}^3 \left| J_{(F_1, F_2, G_1)}^{ii} \right| = \lambda_1\lambda_2 \hfill \\
  = & f_2^2 \left( \dfrac{\partial^2 g_1}{\partial x_2^2 } \dfrac{\partial^2 g_1}{\partial y_1^2} - \left(\dfrac{\partial^2 g_1}{\partial x_2 \partial y_1}\right)^2 \right)\\
  & + f_2 \dfrac{\partial g_1}{\partial x_1} \left( \dfrac{\partial^2 g_1}{\partial y_1^2 }\dfrac{\partial f_1}{\partial x_2} - \dfrac{\partial^2 g_1}{\partial x_2 \partial y_1 } \dfrac{\partial f_1}{\partial y_1} \right).
\end{aligned}
\end{equation}

\smallskip
where $\sigma_1 = Tr(J_{(F_1, F_2, G_1)}) = p$ is the sum of all first-order diagonal minors of $J_{(F_1, F_2, G_1)}$, \textit{i.e.} the trace of $J_{(F_1, F_2, G_1)}$ and $\sigma_2 = \sum_{i=1}^3 \left| J_{(F_1, F_2, G_1)}^{ii} \right| = q $ represents the sum of all second-order diagonal minors of $J_{(F_1, F_2, G_1)}$. Thus, the \textit{pseudo singular point} is of saddle-type \textit{iff} the following conditions $C_1$ and $C_2$ are verified:

\begin{equation}
\label{eq26}
\begin{aligned}
C_1:& \quad \Delta = p^2 - 4q > 0, \hfill \\
C_2:& \quad q < 0.
\end{aligned}
\end{equation}

\smallskip

Condition $C_1$ is systematically satisfied provided that condition $C_2$ is verified. Thus, the \textit{pseudo singular point} is of saddle-type \textit{iff} $q < 0$.

\smallskip

\subsection{Canard existence in $\mathbb{R}^3$}

In an article entitled ``Syst\`{e}mes lents-rapides dans $\mathbb{R}^3$ et leurs canards'', Beno\^{i}t \cite[p. 171]{Benoit1983} has stated in the framework of ``non-standard analysis'' a theorem that can be written as follows:

\begin{flushleft}
\textbf{Beno\^{i}t's theorem [1983]}\hfill \\
\end{flushleft}

\textit{If the desingularized vector field {\rm (17)} has a \textit{pseudo singular point} of saddle type, then system {\rm (11)} exhibits a canard solution which evolves from the attractive part of the slow manifold towards its repelling part.}

\begin{proof}
See Beno\^{i}t [1983].
\end{proof}

In his work, Beno\^{i}t \cite[p. 168]{Benoit1983} computed the trace $T$ and determinant $D$ of the Jacobian matrix $J_{(F_2, G_1)}$ associated with the two-dimensional desingularized vector field {\rm (17)}. Taking into account his generic hypotheses Eqs. (\ref{eq20}-\ref{eq21}) he found that:

\begin{equation}
\label{eq27}
\begin{aligned}
T = & Tr(J_{(F_2, G_1)}) = \lambda_1 + \lambda_2 =  \dfrac{\partial g_1}{\partial x_1}\dfrac{\partial f_1}{\partial y_1}, \hfill  \\
D = & \left| J_{(F_2, G_1)} \right| = \lambda_1\lambda_2 \hfill \\
   = & f_2^2 \left( \dfrac{\partial^2 g_1}{\partial x_2^2 } \dfrac{\partial^2 g_1}{\partial y_1^2} - \left(\dfrac{\partial^2 g_1}{\partial x_2 \partial y_1}\right)^2 \right)\\
   & + f_2 \dfrac{\partial g_1}{\partial x_1} \left( \dfrac{\partial^2 g_1}{\partial y_1^2 }\dfrac{\partial f_1}{\partial x_2} - \dfrac{\partial^2 g_1}{\partial x_2 \partial y_1 } \dfrac{\partial f_1}{\partial y_1} \right).
\end{aligned}
\end{equation}

\smallskip

from which he established that the \textit{pseudo singular point} is of saddle type provided that $D < 0$. Then, Beno\^{i}t \cite[p. 171]{Benoit1983} stated his theorem.\\

In a paper entitled ``Canards et enlacements'', Beno\^{i}t \cite{Benoit1990} stated, while using a standard polynomial change of variables (see Appendix A), that the original system (\ref{eq11}) can be transformed into the following ``normal version'':

\begin{equation}
\label{eq28}
\begin{aligned}
\dot{x}_1 & =  a x_2 + b y_1 + O \left( x_1, \varepsilon, x_2^2, x_2 y_1, y_1^2 \right), \hfill \\
\dot{x}_2 & = 1 + O \left( x_1, x_2, y_1, \varepsilon \right), \hfill \\
\varepsilon \dot{y}_1 & = - \left( x_1 + y_1^2 \right) + O \left( \varepsilon x_1, \varepsilon x_2, \varepsilon y_1, \varepsilon^2, x_1^2 y_1, y_1^3, x_1 x_2 y_1  \right), \hfill
\end{aligned}
\end{equation}

\smallskip

where he established that

\[
\begin{aligned}
a = & \frac{1}{2} f_2^2 \left( \dfrac{\partial^2 g_1}{\partial x_2^2 } \dfrac{\partial^2 g_1}{\partial y_1^2} - \left(\dfrac{\partial^2 g_1}{\partial x_2 \partial y_1}\right)^2 \right)\\
& + \frac{1}{2} f_2 \dfrac{\partial g_1}{\partial x_1} \left( \dfrac{\partial^2 g_1}{\partial y_1^2 }\dfrac{\partial f_1}{\partial x_2} - \dfrac{\partial^2 g_1}{\partial x_2 \partial y_1 } \dfrac{\partial f_1}{\partial y_1} \right), \hfill \\
b & = - \dfrac{\partial g_1}{\partial x_1}\dfrac{\partial f_1}{\partial y_1}, \hfill  \\
\end{aligned}
\]

\smallskip

A few years later, Szmolyan and Wechselberger \cite{SzmolyanWechselberger2001} gave a ``standard version'' of Beno\^{i}t's theorem \cite{Benoit1983} (see Beno\^{i}t's theorem above) for three-dimensional singularly perturbed systems with $k=2$ \textit{slow} variables and $m=1$ \textit{fast}. While using ``standard analysis'' and blow-up technique, Szmolyan and Wechselberger \cite[p. 427]{SzmolyanWechselberger2001} stated in their Lemma 2.1, while using ``a smooth change of coordinates'' (see Appendix A), that the original system (\ref{eq11}) can be transformed into the ``normal form'' (\ref{eq28}) from which they deduced that the condition for the \textit{pseudo singular point} to be of saddle type is $a < 0$. Then, they proved the existence of canard solutions for the original system (\ref{eq11}) according to their Theorem 4.1(a) presented below.

\begin{theo}\hfill \\
\label{theo1}
Assume system {\rm (28)}. In the folded saddle and in the folded node case singular canards perturb to maximal canards solutions for sufficiently small~$\varepsilon$.
\end{theo}

\begin{proof}
See Szmolyan and Wechselberger \cite{SzmolyanWechselberger2001}.
\end{proof}

As previously recalled, the method presented in this paper doesn't use the ``desingularized vector field'' (\ref{eq17}) but the ``normalized slow dynamics'' (\ref{eq16}). So, we have the following proposition:

\smallskip

\begin{prop}\hfill \\
\label{prop1}
If the normalized slow dynamics {\rm (16)} has a \textit{pseudo singular point} of saddle type, \textit{i.e.} if the sum $\sigma_2$ of all second-order diagonal minors of the Jacobian matrix of the normalized slow dynamics {\rm (16)} evaluated at the \textit{pseudo singular point} is negative, \textit{i.e.} if $\sigma_2 < 0$ then, according to Theorem \ref{theo1}, system {\rm (11)} exhibits a canard solution which evolves from the attractive part of the slow manifold towards its repelling part.
\end{prop}

\begin{proof}

According to Eqs. (\ref{eq25},\ref{eq27}) it is easy to verify that:

\begin{equation}
\label{eq29}
\begin{aligned}
\sigma_1 & = Tr(J_{(F_1, F_2, G_1)}) = Tr(J_{(F_2, G_1)}) = T = \lambda_1 + \lambda_2 = -b, \hfill  \\
\sigma_2 & = \sum_{i=1}^3 \left| J_{(F_1, F_2, G_1)}^{ii} \right| = \left| J_{(F_2, G_1)} \right| = D = \lambda_1\lambda_2 = 2a.
\end{aligned}
\end{equation}

\smallskip

So, the condition for which the \textit{pseudo singular point} is of saddle type, \textit{i.e.} $\sigma_2 < 0$ is identical to that proposed by Beno\^{i}t [1983, p. 171] in his theorem, \textit{i.e.} $D < 0$ and also to that provided by Szmolyan and Wechselberger \cite{SzmolyanWechselberger2001}, \textit{i.e.} $a < 0$. So, Prop. \ref{prop1} can be used to state the existence of canard solution for such systems.

\end{proof}

Of course, in the three-dimensional case the proof is obvious. We will see in the next Sect. 4 that for four-dimensional singularly perturbed systems this is not the case. Application of Proposition 1 to the three-dimensional memristor canonical Chua's circuits, presented in Sec. 5, will enable to prove the existence of generic ``canard solutions'' in such Memristor Based Chaotic Circuits.

\section{Four-dimensional singularly perturbed systems}
\label{Sec4}

A four-dimensional \textit{singularly perturbed dynamical system} (\ref{eq2}) with $k=3$ \textit{slow} variables and $m=1$ \textit{fast} may be written as:

\begin{subequations}
\label{eq30}
\begin{align}
 \dot{x}_1 & = f_1 \left( x_1, x_2, x_3, y_1 \right), \hfill \\
 \dot{x}_2 & = f_2 \left( x_1, x_2, x_3, y_1 \right), \hfill \\
 \dot{x}_3 & = f_3 \left( x_1, x_2, x_3, y_1 \right), \hfill \\
 \varepsilon \dot{y}_1 & = g_1 \left( x_1, x_2, x_3, y_1 \right), \hfill
\end{align}
\end{subequations}

\smallskip

$\!$ where $\vec{x}= (x_1, x_2, x_3)^t \in \mathbb{R}^3$, $\vec{y}= (y_1) \in \mathbb{R}$, $0<~\varepsilon~\ll~1$ and the functions $f_i$ and $g_i$ are assumed to be $C^2$ functions of $(x_1, x_2, x_3, y_1)$.

\subsection{Critical Manifold}

The critical manifold equation of system (\ref{eq30}) is defined by setting $\varepsilon = 0$ in Eq. (30d). Thus, we obtain:

\begin{equation}
\label{eq31}
g_1 \left( x_1, x_2, x_3, y_1 \right) = 0.
\end{equation}

By application of the Implicit Function Theorem, let suppose that we can explicitly express from Eq. (\ref{eq31}), say without loss of generality, $x_1$ as functions of the others variables:

\begin{equation}
\label{eq32}
x_1 = \phi_1\left( x_2, x_3, y_1 \right).
\end{equation}

\subsection{Constrained system}

The \textit{constrained system} is obtained by equating to zero the time derivative of $g_1 \left(x_1, x_2, x_3, y_1 \right)$:

\begin{equation}
\label{eq33}
\frac{dg_1}{dt} = \frac{\partial g_1}{\partial x_1} \dot{x}_1 + \frac{\partial g_1}{\partial x_2} \dot{x}_2 + \frac{\partial g_1}{\partial x_3} \dot{x}_3 + \frac{\partial g_1}{\partial y_1} \dot{y}_1 = 0
\end{equation}

By replacing $\dot{x}_i$ by $f_i \left( x_1, x_2, x_3, y_1 \right)$ with $i=1,2,3$, Eqs. (\ref{eq33}) may be written as:

\begin{equation}
\label{eq34}
\dot{y}_1 = - \dfrac{ \dfrac{\partial g_1}{\partial x_1} f_1 + \dfrac{\partial g_1}{\partial x_2} f_2 + \dfrac{\partial g_1}{\partial x_3} f_3 }{ \dfrac{\partial g_1}{\partial y_1} }.
\end{equation}

\smallskip

So, we have the following constrained system:

\begin{equation}
\label{eq35}
\begin{aligned}
 \dot{x}_1 & = f_1 \left( x_1, x_2, x_3, y_1 \right), \hfill \\
 \dot{x}_2 & = f_2 \left( x_1, x_2, x_3, y_1 \right), \hfill \\
 \dot{x}_3 & = f_3 \left( x_1, x_2, x_3, y_1 \right), \hfill \\
 \dot{y}_1 & = - \dfrac{ \dfrac{\partial g_1}{\partial x_1} f_1 + \dfrac{\partial g_1}{\partial x_2} f_2 + \dfrac{\partial g_1}{\partial x_3} f_3 }{ \dfrac{\partial g_1}{\partial y_1} }, \hfill \\
 0 & = g_1 \left( x_1, x_2, x_3, y_1 \right). \hfill
\end{aligned}
\end{equation}

\subsection{Normalized slow dynamics}

By rescaling the time by setting $t = -  \dfrac{\partial g_1}{\partial y_1} \tau$ we obtain the ``normalized slow dynamics'':

\begin{equation}
\label{eq36}
\begin{aligned}
\dot{x}_1 & = - f_1 \left( x_1, x_2, x_3, y_1 \right) \dfrac{\partial g_1}{\partial y_1} = F_1 \left( x_1, x_2, x_3, y_1 \right), \hfill \vspace{6pt} \\
 \dot{x}_2 & = - f_2 \left( x_1, x_2, x_3, y_1 \right) \dfrac{\partial g_1}{\partial y_1} = F_2 \left( x_1, x_2, x_3, y_1 \right), \hfill \vspace{6pt} \\
 \dot{x}_3 & = - f_3 \left( x_1, x_2, x_3, y_1 \right) \dfrac{\partial g_1}{\partial y_1} = F_3 \left( x_1, x_2, x_3, y_1 \right), \hfill \vspace{6pt} \\
 \dot{y}_1 & = \dfrac{\partial g_1}{\partial x_1} f_1 + \dfrac{\partial g_1}{\partial x_2} f_2 + \dfrac{\partial g_1}{\partial x_3} f_3 = G_1 \left( x_1, x_2, x_3, y_1 \right), \hfill \vspace{6pt} \\
 0 & = g_1 \left( x_1, x_2, x_3, y_1 \right). \hfill
\end{aligned}
\end{equation}

\smallskip

where the overdot now denotes the time derivation with respect to $\tau$.

\subsection{Desingularized vector field}

Then, since we have supposed that $x_1$ may be explicitly expressed as a function $\phi_1$ of the others variables (\ref{eq32}), it can be used to project the ``normalized slow dynamics'' (\ref{eq36}) on the tangent bundle of the critical manifold. So, we have:

\begin{equation}
\label{eq37}
\begin{aligned}
 \dot{x}_2 & = - f_2 \left( x_1, x_2, x_3, y_1 \right) \dfrac{\partial g_1}{\partial y_1}, \hfill \vspace{6pt} \\
 \dot{x}_3 & = - f_3 \left( x_1, x_2, x_3, y_1 \right) \dfrac{\partial g_1}{\partial y_1}, \hfill \vspace{6pt}\\
 \dot{y}_1 & = \dfrac{\partial g_1}{\partial x_1} f_1 + \dfrac{\partial g_1}{\partial x_2} f_2 + \dfrac{\partial g_1}{\partial x_3} f_3 . \hfill
\end{aligned}
\end{equation}

\smallskip

in which $x_1$ must be replaced by $\phi_1\left( x_2, x_3, y_1 \right)$.

\subsection{Pseudo singular manifold}

\textit{Pseudo singular manifold} is defined as singular solution of the ``normalized slow dynamics'' (\ref{eq36}), so we have:

\begin{subequations}
\label{eq38}
\begin{align}
 & g_1 \left( x_1, x_2, x_3, y_1 \right) = 0, \hfill \\
 & \dfrac{\partial g_1}{\partial y_1} = 0, \hfill \\
 & \dfrac{\partial g_1}{\partial x_1} f_1 + \dfrac{\partial g_1}{\partial x_2} f_2 + \dfrac{\partial g_1}{\partial x_3} f_3 = 0. \hfill
\end{align}
\end{subequations}

\smallskip

\begin{remark}
In the case of a four-dimensional singularly perturbed system with $k=3$ \textit{slow} variables and $m=1$ \textit{fast}, \textit{pseudo singular manifold} forms a $(k-2)$-dimensional manifold, \textit{i.e.} a $1$-dimensional manifold since the system (\ref{eq38}) comprises $p=3$ equations and $n=4$ variables (unknowns). So, in spite of having a \textit{pseudo singular point} ($\tilde{x}_{1}, \tilde{x}_{2}, \tilde{x}_{3}, \tilde{y}_{1}$) we have \textit{pseudo singular manifold} represented by, say without loss of generality, ($\tilde{x}_{1},x_2, \tilde{x}_{3}, \tilde{y}_{1}$), where $x_2$ is undetermined.\\

Let's notice again that contrary to the previous works we don't use the ``desingularized vector field'' (\ref{eq37}) but the ``normalized slow dynamics'' (\ref{eq36}).
\end{remark}

The Jacobian matrix of system (\ref{eq36}) reads:

\begin{equation}
\label{eq39}
J_{(F_1, F_2, F_3, G_1)} = \begin{pmatrix}
\dfrac{\partial F_1}{\partial x_1 } \  &  \  \dfrac{\partial F_1}{\partial x_2 } \  &  \  \dfrac{\partial F_1}{\partial x_3 } \  &  \  \dfrac{\partial F_1}{\partial y_1 } \vspace{6pt} \\
\dfrac{\partial F_2}{\partial x_1 } \  &  \  \dfrac{\partial F_2}{\partial x_2 } \  &  \  \dfrac{\partial F_2}{\partial x_3 } \  &  \  \dfrac{\partial F_2}{\partial y_1 } \vspace{6pt} \\
\dfrac{\partial F_3}{\partial x_1 } \  &  \  \dfrac{\partial F_3}{\partial x_2 } \  &  \  \dfrac{\partial F_3}{\partial x_3 } \  &  \  \dfrac{\partial F_3}{\partial y_1 } \vspace{6pt} \\
\dfrac{\partial G_1}{\partial x_1 } \  &  \  \dfrac{\partial G_1}{\partial x_2 } \  &  \  \dfrac{\partial G_1}{\partial x_3 } \  &  \  \dfrac{\partial G_1}{\partial y_1 } \vspace{6pt}
\end{pmatrix}
\end{equation}

\smallskip

\subsection{Extension of Beno\^{i}t's generic hypothesis}

Without loss of generality, it seems reasonable to extend Beno\^{i}t's generic hypotheses introduced for the three-dimensional case to the four-dimensional case. So, first let's suppose that by a ``standard translation'' the \textit{pseudo singular manifold} can be transformed into ($0,x_2, 0, 0$) and that by a ``standard rotation'' of $y_1$-axis that the \textit{critical manifold} (\ref{eq31}) is tangent to ($x_2, x_3, y_1$)-hyperplane, so we have

\begin{equation}
\label{eq40}
\begin{aligned}
& f_1 \left(0, x_2, 0, 0 \right) = g_1 \left( 0, x_2, 0, 0 \right) = 0 \\
& \left.  \dfrac{\partial g_1}{\partial x_2} \right|_{(0,x_2,0,0)} = \left.  \dfrac{\partial g_1}{\partial x_3} \right|_{(0,x_2,0,0)} = \left.  \dfrac{\partial g_1}{\partial y_1} \right|_{(0,x_2,0, 0)} = 0
\end{aligned}
\end{equation}

Then, let's make the following assumptions for the non-degeneracy of the \textit{folded singularity}:

\begin{equation}
\label{eq41}
f_2 \left(0, x_2, 0, 0 \right) \neq 0 \mbox{ ; } \left.  \dfrac{\partial g_1}{\partial x_1} \right|_{(0,x_2,0,0)} \neq 0 \mbox{ ; } \left.  \dfrac{\partial^2 g_1}{\partial y_1^2} \right|_{(0,x_2,0,0)} \neq 0.
\end{equation}

\smallskip

According to these generic hypotheses Eqs. (\ref{eq40}-\ref{eq41}), the Jacobian matrix (\ref{eq39}) reads:

\begin{equation}
\label{eq42}
J_{(F_1, F_2, F_3, G_1)} = \begin{pmatrix}
0  &  0  &  0  &  0  \vspace{6pt} \\
a_{21}  &  a_{22}   &  a_{23}  &  a_{24} \vspace{6pt} \\
a_{31}  &  a_{32}   &  a_{33}  &  a_{34} \vspace{6pt} \\
a_{41}  &  a_{42}   &  a_{43}  &  a_{44} \vspace{6pt}
\end{pmatrix}
\end{equation}

\smallskip

where

\smallskip

\[
\begin{aligned}
a_{2i} = & -f_2 \dfrac{\partial^2 g_1}{\partial x_i \partial y_1} \mbox{ for } i =1,2,3,\\
a_{24} = & - f_2 \dfrac{\partial^2 g_1}{\partial y_1^2} \\
a_{3i} = & -f_3 \dfrac{\partial^2 g_1}{\partial x_i \partial y_1} \mbox{ for } i =1,2,3,\\
a_{34} = & - f_3 \dfrac{\partial^2 g_1}{\partial y_1^2} \\
a_{4i} = & f_1 \dfrac{\partial^2 g_1}{\partial x_1 \partial x_i} + \dfrac{\partial g_1}{\partial x_1}\dfrac{\partial f_1}{\partial x_i} + f_2 \dfrac{\partial^2 g_1}{\partial x_2 \partial x_i} + \dfrac{\partial g_1}{\partial x_2}\dfrac{\partial f_2}{\partial x_i}\\
&+ f_3 \dfrac{\partial^2 g_1}{\partial x_3 \partial x_i} + \dfrac{\partial g_1}{\partial x_3}\dfrac{\partial f_3}{\partial x_i} \mbox{ for } i =1,2,3,\\
a_{44} = & f_1 \dfrac{\partial^2 g_1}{\partial x_1 \partial y_1} + \dfrac{\partial g_1}{\partial x_1}\dfrac{\partial f_1}{\partial y_1} + f_2 \dfrac{\partial^2 g_1}{\partial x_2 \partial y_1} + \dfrac{\partial g_1}{\partial x_2}\dfrac{\partial f_2}{\partial y_1}\\
& + f_3 \dfrac{\partial^2 g_1}{\partial x_3 \partial y_1} + \dfrac{\partial g_1}{\partial x_3}\dfrac{\partial f_3}{\partial y_1}.
\end{aligned}
\]

\smallskip

In his paper Wechselberger \cite{Wechselberger2012} stated that the determinant of the Jacobian matrix associated to the ``desingularized vector field'' and evaluated at a \textit{folded singularity}, \textit{i.e.} on the \textit{pseudo singular manifold} is always zero\footnote{This result will be proved below.}.

\smallskip

Thus, we have the following Cayley-Hamilton eigenpolynomial associated with such a Jacobian matrix (\ref{eq41}) evaluated on the \textit{pseudo singular manifold}, \textit{i.e.}, at ($0, x_2, 0, 0$):

\begin{equation}
\label{eq43}
\lambda^4 - \sigma_1 \lambda^3 + \sigma_2 \lambda^2 - \sigma_3 \lambda + \sigma_4 = 0,
\end{equation}

\smallskip
where $\sigma_1 = Tr(J_{(F_1, F_2, F_3, G_1)} )$ is the sum of all first-order diagonal minors of $J_{(F_1, F_2, F_3, G_1)} $, \textit{i.e.}, the trace of $J_{(F_1, F_2, F_3, G_1)} $, $\sigma_2$ represents the sum of all second-order diagonal minors of $J_{(F_1, F_2, F_3, G_1)} $ and $\sigma_3$ represents the sum of all third-order diagonal minors of $J_{(F_1, F_2, F_3, G_1)} $. It appears that $\sigma_4 = |J_{(F_1, F_2, F_3, G_1)} |=0$ since one row of the Jacobian matrix (\ref{eq42}) is null. So, the Cayley-Hamilton eigenpolynomial reduces to:

\begin{equation}
\label{eq44}
\lambda\left( \lambda^3 - \sigma_1 \lambda^2 + \sigma_2 \lambda - \sigma_3 \right) = 0.
\end{equation}

\smallskip

But, according to Wechselberger \cite{Wechselberger2012}, $\sigma_3$ vanishes on the \textit{pseudo singular manifold}. Let's prove it:

\begin{proof}

The sum of all third-order diagonal minors of $J$ reads:

\[
\sigma_3 =
\begin{vmatrix}
-f_2 \dfrac{\partial^2 g_1}{\partial x_2 \partial y_1} \quad  & \quad -f_2 \dfrac{\partial^2 g_1}{\partial x_3 \partial y_1} \quad  &  \quad - f_2 \dfrac{\partial^2 g_1}{\partial y_1^2} \vspace{6pt} \\
-f_3 \dfrac{\partial^2 g_1}{\partial x_2 \partial y_1} \quad  & \quad -f_3 \dfrac{\partial^2 g_1}{\partial x_3 \partial y_1} \quad  &  \quad - f_3 \dfrac{\partial^2 g_1}{\partial y_1^2} \vspace{6pt} \\
a_{42} \quad  &  \quad a_{43} \quad & \quad a_{44} \vspace{6pt}
\end{vmatrix}
\]

\smallskip

Then, while using a Laplace's expansion to compute this determinant, it's easy to show that it vanishes.

\end{proof}

So, the Cayley-Hamilton eigenpolynomial (\ref{eq44}) is thus reduced to

\begin{equation}
\label{eq45}
\lambda^2\left( \lambda^2 - \sigma_1 \lambda + \sigma_2 \right) = 0
\end{equation}

\smallskip

Let $\lambda_i$ be the eigenvalues of the eigenpolynomial (\ref{eq45}) and let's denote by $\lambda_{3,4} = 0$ the obvious double root of this polynomial. We have:

\begin{equation}
\label{eq46}
\begin{aligned}
\sigma_1 = & Tr(J_{(F_1, F_2, F_3, G_1)} ) = \lambda_1 + \lambda_2 = - \dfrac{\partial g_1}{\partial x_1}\dfrac{\partial f_1}{\partial y_1}, \hfill  \\
\sigma_2 = & \sum_{i=1}^3 \left| J_{(F_1, F_2, F_3, G_1)}^{ii} \right| =  \lambda_1\lambda_2 \hfill \\
 = & 2 f_2 f_3 \left( \dfrac{\partial^2 g_1}{\partial x_2 \partial x_3 } \dfrac{\partial^2 g_1}{\partial y_1^2} - \dfrac{\partial^2 g_1}{\partial x_2 \partial y_1}\dfrac{\partial^2 g_1}{\partial x_3 \partial y_1} \right) \hfill \\
& + f_2^2 \left( \dfrac{\partial^2 g_1}{\partial x_2^2 } \dfrac{\partial^2 g_1}{\partial y_1^2} - (\dfrac{\partial^2 g_1}{\partial x_2 \partial y_1})^2 \right) \hfill \\
& +  f_2 \dfrac{\partial g_1}{\partial x_1} \left( \dfrac{\partial^2 g_1}{\partial y_1^2 }\dfrac{\partial f_1}{\partial x_2} - \dfrac{\partial^2 g_1}{\partial x_2 \partial y_1 } \dfrac{\partial f_1}{\partial y_1} \right) \hfill \\
& +  f_3^2 \left( \dfrac{\partial^2 g_1}{\partial x_3^2 } \dfrac{\partial^2 g_1}{\partial y_1^2} - (\dfrac{\partial^2 g_1}{\partial x_3 \partial y_1})^2 \right) \hfill \\
& + f_3 \dfrac{\partial g_1}{\partial x_1} \left( \dfrac{\partial^2 g_1}{\partial y_1^2 }\dfrac{\partial f_1}{\partial x_3} - \dfrac{\partial^2 g_1}{\partial x_3 \partial y_1 } \dfrac{\partial f_1}{\partial y_1} \right), \hfill \\
\end{aligned}
\end{equation}

\smallskip

where $\sigma_1 = Tr(J_{(F_1, F_2, F_3, G_1)} ) = p$ is is the sum of all first-order diagonal minors of $J_{(F_1, F_2, F_3, G_1)}$, \textit{i.e.} the trace of the Jacobian matrix $J_{(F_1, F_2, F_3, G_1)} $ and $\sigma_2 = \sum_{i=1}^3 \left| J_{(F_1, F_2, F_3, G_1)}^{ii} \right| = q$ represents the sum of all second-order diagonal minors of $J_{(F_1, F_2, F_3, G_1)} $.

Thus, the \textit{pseudo singular manifold} is of saddle-type \textit{iff} the following conditions $C_1$ and $C_2$ are verified:

\begin{equation}
\label{eq47}
\begin{aligned}
C_1:& \quad \Delta = p^2 - 4q > 0, \hfill \\
C_2:& \quad q < 0.
\end{aligned}
\end{equation}

\smallskip

Condition $C_1$ is systematically satisfied provided that condition $C_2$ is verified. Thus, the \textit{pseudo singular manifold} is of saddle-type \textit{iff} $q < 0$. But, as recalled previously, one coordinate is undetermined, say $x_2$ without loss of generality. So, the eigenvalues (\ref{eq46}) of the characteristic polynomial are also functions of the variable $x_2$ and of the parameters of system (\ref{eq30}). Now, let suppose that one parameter, say without loss of generality $\alpha_2$ (see Sec. 6), modifies the nature of the \textit{pseudo singular manifold}. Condition $C_2$, \textit{i.e.} $q < 0$ is then represented in the space ($x_2, \alpha_2$) by a straight line defining a region within which the \textit{pseudo singular points} are of saddle type. In other words, it means that by choosing a value of the coordinate $x_2$ inside this region ensures that the \textit{pseudo singular point} would be of saddle type.

\subsection{Canard existence in $\mathbb{R}^4$}

In a paper entitled ``A propos de canards'' Wechselberger \cite{Wechselberger2012} stated, while using a standard polynomial change of variables, that any $n$-dimensional singularly perturbed systems with $k$ \textit{slow} variables ($k \geqslant 2$) and $m$ \textit{fast} ($m \geqslant 1$) (\ref{eq1}) can be transformed into the following ``normal form''  (see Appendix B):

\begin{equation}
\label{eq48}
\begin{aligned}
 \dot{x}_1 = &   \tilde{a} x_2 + \tilde{b} y_1 + O \left( x_1, x_2^2, x_2 y_1, y_1^2 \right) + \varepsilon O \left( x_1,  x_2 , x_k, y_1 \right), \hfill \\
 \dot{x}_2 = & 1 + O \left( x_1, x_2, y_1, \varepsilon \right), \hfill \\
 \dot{x}_j = & \tilde{c}_j + O \left( x_1, x_2, y_1, \varepsilon \right), \quad j = 3, \dots, k \hfill \\
 \varepsilon \dot{y}_1 = & x_1 + y_1^2 + x_1 y_1 O \left(x_2, \ldots, x_k  \right) + y_1^2 O \left(x_1, y_1  \right) + \varepsilon O \left( x_1, x_2, y_1, \varepsilon \right) \hfill
\end{aligned}
\end{equation}

which is a generalization of system (\ref{eq28}). We will establish in Appendix B for any four-dimensional singularly perturbed systems (\ref{eq30}) with $k=3$ \textit{slow} variables and $m=1$ \textit{fast} variable that

\[
\begin{aligned}
\tilde{a} & = \frac{1}{2} f_2^2 \left( \dfrac{\partial^2 g_1}{\partial x_2^2 } \dfrac{\partial^2 g_1}{\partial y_1^2} - (\dfrac{\partial^2 g_1}{\partial x_2 \partial y_1})^2 \right) + \frac{1}{2} f_2 \dfrac{\partial g_1}{\partial x_1} \left( \dfrac{\partial^2 g_1}{\partial y_1^2 }\dfrac{\partial f_1}{\partial x_2} - \dfrac{\partial^2 g_1}{\partial x_2 \partial y_1 } \dfrac{\partial f_1}{\partial y_1} \right) \hfill \\
& + \frac{1}{2} f_3^2 \left( \dfrac{\partial^2 g_1}{\partial x_3^2 } \dfrac{\partial^2 g_1}{\partial y_1^2} - (\dfrac{\partial^2 g_1}{\partial x_3 \partial y_1})^2 \right) + \frac{1}{2} f_3 \dfrac{\partial g_1}{\partial x_1} \left( \dfrac{\partial^2 g_1}{\partial y_1^2 }\dfrac{\partial f_1}{\partial x_3} - \dfrac{\partial^2 g_1}{\partial x_3 \partial y_1 } \dfrac{\partial f_1}{\partial y_1} \right) \hfill \\
& + f_2 f_3 \left( \dfrac{\partial^2 g_1}{\partial x_2 \partial x_3 } \dfrac{\partial^2 g_1}{\partial y_1^2} - \dfrac{\partial^2 g_1}{\partial x_2 \partial y_1}\dfrac{\partial^2 g_1}{\partial x_3 \partial y_1} \right), \hfill \\
\tilde{b} & = - \dfrac{\partial g_1}{\partial x_1}\dfrac{\partial f_1}{\partial y_1}, \hfill  \\
\end{aligned}
\]

\begin{remark}
Let's notice that by posing $f_3=0$ in $\tilde{a}$ we find again $a$ given in Sec. 3.7.
\end{remark}

Thus, in his article entitled ``A propos de canards'' Wechselberger \cite[p. 3304]{Wechselberger2012} has provided in the framework of ``standard analysis'' a generalization of Beno\^{i}t's theorem \cite{Benoit1983} (see Beno\^{i}t's theorem above) for any $n$-dimensional singularly perturbed systems with $k$ \textit{slow} variables ($k \geqslant 2$) and $m$ \textit{fast} ($m \geqslant 1$). According to his Theorem 4.1(b) presented below he proved the existence of canard solutions for the original system (\ref{eq1}).

\begin{theo}\hfill \\
\label{theo2}
In the folded saddle case of system {\rm (48)} singular canards perturb to maximal canards solutions for sufficiently small $\varepsilon \ll 1$.
\end{theo}

\begin{proof}
See Wechselberger \cite{Wechselberger2012}.
\end{proof}

As previously recalled, the method presented in this paper doesn't use the ``desingularized vector field'' (\ref{eq37}) but the ``normalized slow dynamics'' (\ref{eq36}). So, we have the following proposition:

\smallskip

\begin{prop}\hfill \\
\label{prop2}
If the normalized slow dynamics {\rm (36)} has a \textit{pseudo singular point} of saddle type, \textit{i.e.} if the sum $\sigma_2$ of all second-order diagonal minors of the Jacobian matrix of the normalized slow dynamics {\rm (36)} evaluated at a \textit{pseudo singular point} is negative, \textit{i.e.} if $\sigma_2 < 0$ then, according to Theorem 2, system {\rm (30)} exhibits a canard solution which evolves from the attractive part of the slow manifold towards its repelling part.

\end{prop}

\begin{proof}

By making some smooth changes of time and smooth changes of coordinates (see Appendix B) we brought the system (\ref{eq30}) to the following ``normal form'':

\[
\begin{aligned}
 \dot{x}_1 & =  \tilde{a} x_2 + \tilde{b} y_1 + O \left( x_1, \varepsilon, x_2^2, x_2 y_1, y_1^2 \right), \hfill \\
 \dot{x}_2 & = 1 + O \left( x_1, x_2, y_1, \varepsilon \right), \hfill \\
 \dot{x}_3 & = 1 + O \left( x_1, x_2, y_1, \varepsilon \right), \hfill \\
 \varepsilon \dot{y}_1 & = x_1 + y_1^2 + O \left( \varepsilon x_1, \varepsilon x_2, \varepsilon y_1, \varepsilon^2, x_1^2 y_1, y_1^3, x_1 x_2 y_1  \right), \hfill
\end{aligned}
\]

\smallskip

Then, we deduce that the condition for the \textit{pseudo singular point} to be of saddle type is $ \tilde{a} < 0$. According to Eqs. (\ref{eq46}) it is easy to verify that

\[
\begin{aligned}
\sigma_1 & = Tr(J_{(F_1, F_2, F_3, G_1)}) = \lambda_1 + \lambda_2 = -\tilde{b}, \hfill  \\
\sigma_2 & = \sum_{i=1}^3 \left| J_{(F_1, F_2, F_3, G_1)}^{ii} \right| = \lambda_1\lambda_2 = 2\tilde{a}.
\end{aligned}
\]

\smallskip

So, the condition for which the \textit{pseudo singular point} is of saddle type, \textit{i.e.} $\sigma_2 < 0$ is identical to that proposed by Wechselberger \cite[p. 3298]{Wechselberger2012} in his theorem, \textit{i.e.} $\tilde{a} < 0$.
\end{proof}

So, Prop. 2 can be used to state the existence of canard solution for such systems. Application of Proposition 2 to the four-dimensional memristor canonical Chua's circuits, presented in Sec. 6, will enable to prove the existence of generic ``canards solutions'' in such Memristor Based Chaotic Circuits.

\section{Third-order Memristor-Based canonical oscillator}
\label{Sec5}

Let's consider the Memristor-Based canonical Chua's circuit \cite{ItohChua2008,ItohChua2013} containing five circuits elements: two passive capacitors, one passive inductor, one negative resistor, and one active Chua's flux controlled memristor (see Fig. 2).

\begin{figure}[htbp]
\centerline{\includegraphics[width = 8.5cm,height = 4.25cm]{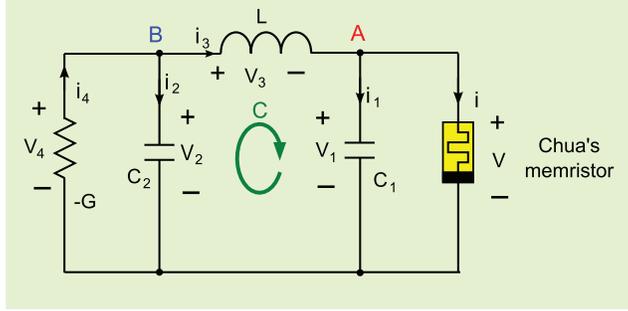}}
\caption{Memristor-Based canonical Chua's circuit \cite{ItohChua2013}.}
\label{Fig2}
\end{figure}

The parameter values used by Itoh and Chua \cite[p. 1330001-14]{ItohChua2013} \textit{i.e.} are:

\[
C_1 = \dfrac{1}{10}\mbox{, } C_2 = \dfrac{1}{0.47}\mbox{, } G = L = 1\mbox{, } a = -2.0\mbox{, } b = 4.0.
\]

\subsection{Flux-linkage and charge phase space}

Applying Kirchhoff's circuit laws to the nodes $A$, $B$ and the loop $C$ of the circuit Fig. 2, Itoh and Chua \cite{ItohChua2008,ItohChua2013} obtained the following set of differential equations, \textit{i.e.}, the following \textit{memristor based chaotic circuit}:

\begin{equation}
\label{eq49}
\begin{aligned}
C_1 \dfrac{d\varphi_1}{dt} & = q_3 - k\left( \varphi_1 \right), \hfill \\
C_2 \dfrac{d\varphi_2}{dt} & = - q_3 + G \varphi_2, \hfill \\
 L \dfrac{dq_3}{dt} & = \varphi_2 - \varphi_1.
\end{aligned}
\end{equation}

where the $\varphi-q$ characteristic curve of the Chua's memristor is given by the following piecewise-linear function:

\begin{equation}
\label{eq50}
q = k\left( \varphi \right) = b \varphi + \frac{a -b}{2} \left( \left| \varphi + 1 \right| - \left| \varphi - 1 \right| \right)
\end{equation}

\vspace{0.1in}

By setting $x= \varphi_1$, $y = q_3$, $z = \varphi_2$, $\varepsilon = C_1$, $\beta = \dfrac{1}{C_2}$, $\gamma = \dfrac{G}{C_2}$ and $L =1$ the \textit{memristor based chaotic circuit} (\ref{eq50}) can be written:

\begin{equation}
\label{eq51}
\begin{aligned}
\dfrac{dx}{dt} & = \dfrac{1}{\varepsilon} \left[ y - k\left( x \right) \right], \hfill \\
\dfrac{dy}{dt} & = z - x, \hfill \\
\dfrac{dz}{dt} & = - \beta y + \gamma z.
\end{aligned}
\end{equation}

\vspace{0.1in}

Following the works of Tsuneda \cite{Tsuneda2005}, let's replace the $\varphi-q$ characteristic curve of the Chua's memristor $q( \varphi )$ which is given by the piecewise-linear function (\ref{eq51}) by a smooth cubic nonlinear function $\hat{k}( \varphi ) = c_1 \varphi^3 + c_2 \varphi$ for which the parameters $c_1$ and $c_2$ are determined while using the \textit{least squares method}. The square error between $k( \varphi )$ and $\hat{k}( \varphi )$ is defined by:

\begin{equation}
\label{eq52}
S = \int_{-d}^{d} \left[ k( \varphi ) - \hat{k}( \varphi ) \right]^2 \mathrm{d}\varphi
\end{equation}

\vspace{0.1in}

where $\left[-d, d \right]$ is an interval for approximation. Let's note that in our case $d$ is considered as a parameter such that $|d| > 1$. Solving $\partial S / \partial c_1 = 0$ and $\partial S / \partial c_2 = 0$, we find

\begin{equation}
\label{eq53}
\begin{aligned}
c_1 & = -\frac{35 (a-b) \left(-1+d^2\right)^2}{16 d^7}, \hfill\\
c_2 & = \frac{(a-b) \left(21-50 d^2+45 d^4\right)}{16 d^5} + b.
\end{aligned}
\end{equation}

\subsection{Piecewise-linear and cubic nonlinearity}

While still using the same parameter values as Itoh and Chua \cite[p. 1330001-14]{ItohChua2013} \textit{i.e.}

\[
C_1 = \dfrac{1}{10}\mbox{, } C_2 = \dfrac{1}{0.47}\mbox{, } G = L = 1\mbox{, } a = -2.0\mbox{, } b = 4.0,
\]

the coefficients $c_1$ and $c_2$ have been chosen such that the extrema of both piecewise-linear and cubic nonlinearity characteristic curves substantially coincides as exemplified on Fig. 3. This condition is realized for $d=3$ and

\[
c_1 = \frac{280}{729} \quad \mbox{ ; } \quad c_2 = -\frac{26}{27}.
\]

\begin{figure}[htbp]
\centerline{\includegraphics[width = 11cm,height = 11cm]{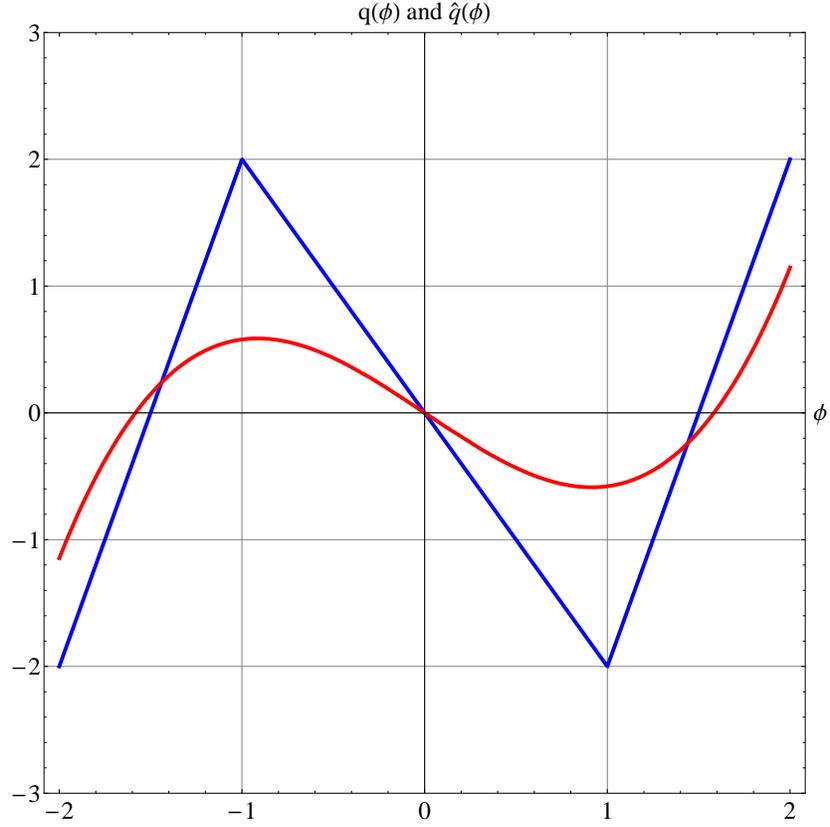}}
\caption{Piecewise-linear and cubic $\varphi-q$ characteristic curves for parameter values: $a = -2$, $b = 4$ and $d=3$.}
\label{fig3}
\end{figure}

So, let's consider the \textit{memristor based chaotic circuit} (\ref{eq51}):

\begin{equation}
\label{eq54}
\begin{aligned}
\dfrac{dx}{dt} & = \dfrac{1}{\varepsilon} \left[ y - k\left( x \right) \right], \hfill \\
\dfrac{dy}{dt} & = z - x, \hfill \\
\dfrac{dz}{dt} & = - \beta y + \gamma z,
\end{aligned}
\end{equation}

\vspace{0.1in}

and let's replace the piecewise-linear characteristic curves $k(x)$ by the cubic $\hat{k}(x) = c_1 x^3 + c_2 x$. First, let's  notice that both \textit{chaotic attractors} given respectively by Eqs. (\ref{eq51}) \& Eqs. (\ref{eq54}) are quite \textit{similar} as highlighted on Fig. 4.

\begin{figure}[htbp]
\centerline{\includegraphics[width = 11cm,height = 11cm]{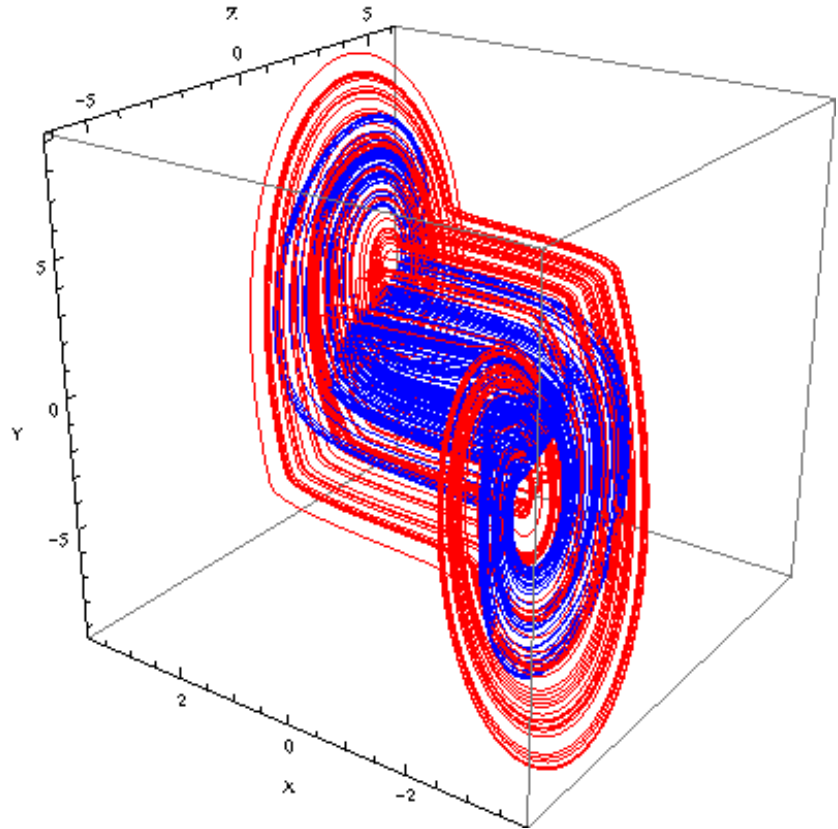}}
\vspace{0.1in}
\caption{\textit{Memristor-Based canonical Chua's circuits} with piecewise linear (Eqs. (\ref{eq51}) in red) and cubic (Eqs. (\ref{eq55}) in blue) functions for parameter values: $\varepsilon = 1/10$, $\beta = \gamma = 0.47$, $a = -2$, $b = 4$ and $d=3$.}
\label{fig4}
\end{figure}

Now, let's make the following variable changes in Eqs. (\ref{eq54}) in order to apply the method presented in Sec. 3:

\[
x \rightarrow z, \quad y \rightarrow -x, \quad z \rightarrow y.
\]

Thus, we have:

\begin{equation}
\label{eq55}
\begin{aligned}
\dfrac{dx}{dt} & = z - y, \hfill \\
\dfrac{dy}{dt} & = \beta x + \gamma y, \hfill \\
\dfrac{dz}{dt} & = \dfrac{1}{\varepsilon} \left[- x - k\left( z \right) \right].
\end{aligned}
\end{equation}

\vspace{0.1in}

This last transformation will enable to compare the condition (given below) for the existence of canard solutions in system (\ref{eq55}) with those given in our previous works entitled ``Canards from Chua's circuits'' \cite{GinouxLLibre2013}.\\

Finally, let's replace the variables ($x,y,z$) by ($x_1, x_2, y_1$) and let's apply the method presented in Sec. 3 to the following system

\begin{equation}
\label{eq56}
\begin{aligned}
\dfrac{dx_1}{dt} & = y_1 - x_2, \hfill \\
\dfrac{dx_2}{dt} & = \beta x_1 + \gamma x_2, \hfill \\
\dfrac{dy_1}{dt} & = \dfrac{1}{\varepsilon} \left[- x_1 - k\left( y_1 \right) \right].
\end{aligned}
\end{equation}

\subsection{Critical manifold and constrained system}

The critical manifold of this system (\ref{eq56}) is given by $-x_1 - k(y_1) = 0$. According to Eq. (\ref{eq15}) the constrained system on the critical manifold reads:

\begin{equation}
\label{eq57}
\begin{aligned}
 \dot{x_1} & = y_1 - x_2, \hfill \\
 \dot{x_2} & = \beta x_1 + \gamma x_2, \hfill \\
 \dot{y_1} & = \dfrac{ y_1 - x_2 }{- \left(c_1 y_1^3 + c_2 y_1 \right)}, \hfill \vspace{6pt} \\
         0 & = - x_1 - \left(c_1 y_1^3 + c_2 y_1 \right).
\end{aligned}
\end{equation}

\subsection{Normalized slow dynamics}

Then, by rescaling the time by setting $t = -  \dfrac{\partial g_1}{\partial y_1} \tau = (3 c_1 y_1^2 + c_2) $ we obtain the ``normalized slow dynamics'':

\begin{equation}
\label{eq58}
\begin{aligned}
 \dot{x}_1 & = \left( y_1 - x_2 \right) \left( 3 c_1 y_1^2 + c_2 \right) = F_1 \left( x_1, x_2, y_1 \right), \hfill \vspace{6pt} \\
 \dot{x}_2 & = \left( \beta x_1 + \gamma x_2 \right) \left( 3 c_1 y_1^2 + c_2 \right) = F_2 \left( x_1, x_2, y_1 \right), \hfill \vspace{6pt} \\
 \dot{y}_1 & = x_2 - y_1 = G_1 \left( x_1, x_2, y_1 \right), \hfill \vspace{6pt} \\
 0 & = - x_1 - \left(c_1 y_1^3 + c_2 y_1 \right). \hfill
\end{aligned}
\end{equation}

\subsection{Pseudo singular points}

According to Eq. (\ref{eq18}), the \textit{pseudo singular points} of system (\ref{eq56}) are:

\begin{equation}
\label{eq59}
\tilde{x}_1 = \pm \dfrac{2c_2}{3} \sqrt{ \dfrac{-c_2}{3c_1}}, \ \tilde{x}_2 = \mp \sqrt{ \dfrac{- c_2}{3c_1}}, \ \tilde{y}_1 = \mp \sqrt{ \dfrac{-c_2}{3c_1}}.
\end{equation}

\smallskip

Let's notice that these \textit{pseudo singular points} are independent of the parameter $\gamma$. The Jacobian matrix of system (\ref{eq58}) evaluated at the \textit{pseudo singular points} reads:

\begin{equation}
\label{eq60}
J_{(F_1, F_2, G_1)} = \begin{pmatrix}
0 \quad &  \quad 0 \quad &  \quad 0 \vspace{6pt} \\
0 \quad &  \quad 0 \quad &  \quad -2 \gamma  c_2 + \frac{4 \beta  c_2^2}{3} \vspace{6pt} \\
0 \quad &  \quad 1 \quad &  \quad -1 \vspace{6pt}
\end{pmatrix}
\end{equation}

\begin{remark}
Although, the \textit{pseudo singular points} have not been shifted at the origin Beno\^{i}t's generic hypotheses (\ref{eq20}-\ref{eq21}) are satisfied.
\end{remark}

\subsection{Canard existence in third-order memristor Chua's circuit}

According to Eqs. (\ref{eq25}) we find that:

\[
\begin{aligned}
p & = \sigma_1 = Tr\left[ J\right] = -1,\\
q & = \sigma_2 = \frac{2}{3} c_2 \left(3 \gamma -2 \beta  c_2\right).
\end{aligned}
\]

Thus, according to Prop. \ref{prop1}, the \textit{pseudo singular points} are of saddle-type if and only if:

\[
\frac{2}{3} c_2 \left(3 \gamma -2 \beta  c_2\right) < 0.
\]

\[
\Delta = p^2 - 4q > 0 \qquad \mbox{ and } \qquad q <0.
\]

So, we have the following conditions $C_1$ and $C_2$:

\begin{equation}
\label{eq61}
\begin{aligned}
C_1:& \quad \Delta = 1 + 4 (- 2 c_2) ( \gamma - \dfrac{2 \beta c_2}{3} ) > 0, \hfill \\
C_2:& \quad q = 2c_2( \gamma - \dfrac{2 \beta c_2}{3} )  < 0.
\end{aligned}
\end{equation}

\smallskip
Since the \textit{pseudo singular points} are independent of the parameter $\gamma$ let's choose $\gamma$ as the ``canard parameter'' or ``duck parameter''. Obviously, it appears that if the condition $C_2$ is verified then the condition $C_1$ is \textit{de facto} satisfied\footnote{Keep in mind that $c_2$ is generally negative so that the characteristic curve admits a negative slope.}. Finally, the \textit{pseudo singular points} are of saddle-type if and only if we have:

\begin{equation}
\label{eq62}
\gamma_{saddle-node} = \dfrac{ 2 \beta c_2}{3} < \gamma.
\end{equation}

where $\gamma_{saddle-node}$ represents the critical value of the parameter $\gamma$ for which one of the two remaining eigenvalues $\lambda_1$ or $\lambda_2$ of the eigenpoynomial associated with the Jacobian matrix (\ref{eq60}) vanishes. With this set of parameters $\varepsilon = 1/10$, $\beta = 0.47$, $a = -2$, $b = 4$, $d=3$, $c_1 = 280/729$, $c_2 = -26/27$

\[
\gamma_{saddle-node} = \dfrac{ 2 \beta c_2}{3} \approx -0.3.
\]

\subsection{Fixed points stability and Routh-Hurwitz' theorem}

However, as pointed out in our previous works entitled ``Canards from Chua's circuits'' \cite{GinouxLLibre2013} the system (\ref{eq56}) admits, except the origin, two fixed points, the stability of which could preclude the existence of ``canards solutions''. So, let's compute the fixed points of system (\ref{eq56}) and analyze their stability. System (\ref{eq56}) admits except the origin the following fixed points:

\begin{equation}
\label{eq63}
x_1^* = \pm \dfrac{\gamma}{\beta} \sqrt{ \dfrac{\gamma - c_2 \beta}{c_1 \beta}}, \quad x_2^* = y_1^* = \mp \sqrt{ \dfrac{\gamma - c_2 \beta}{c_1 \beta}}.
\end{equation}

\smallskip

The \textit{eigenpolynomial} equation of the Jacobian matrix of system (\ref{eq56}) evaluated at these fixed points (\ref{eq63}) reads:

\begin{equation}
\label{eq64}
\varepsilon  \lambda^3  + \lambda^2 (\dfrac{3 \gamma }{\beta } - \gamma  \varepsilon - 2 c_2) + \lambda  (1-\dfrac{3 \gamma^2}{\beta } + \beta \varepsilon + 2 \gamma  c_2 ) + 2 ( \gamma - \beta  c_2 ) = 0
\end{equation}

\smallskip

Let suppose that all the parameters are fixed except $\gamma$, \textit{i.e.} the ``duck parameter''. There are two methods to analyze the stability of fixed points as functions of the ``duck parameter'' value. The first is to solve the above third degree eigenpolynomial equation (\ref{eq64}) with the Cardano's method and the second consists in using the Routh-Hurwitz' theorem \cite{Routh1877,Hurwitz1893}. This latter method enables to easier analyze the stability of the fixed points as functions of a parameter. According to Routh-Hurwitz' theorem, the \textit{eigenpolynomial} equation can be written as:

\[
a_3 \lambda^3 + a_2 \lambda^2 + a_1 \lambda + a_0 = 0.
\]

\smallskip

It states that if $D_1 = a_1$ and $D_2 = a_1a_2 - a_0 a_3$ are both positive then \textit{eigenpolynomial} equation would have eigenvalues with negative real parts. In other words, if $D_1$ and $D_2$ are positive the fixed points will be stable. In the case of the \textit{eigenpolynomial} equation (\ref{eq64}) we have:

\begin{equation}
\label{eq65}
\begin{aligned}
D_1 = & 1-\dfrac{3 \gamma^2}{\beta } + \beta \varepsilon + 2 \gamma  c_2, \hfill \\
D_2 = & - 2 \varepsilon  \left(\gamma -\beta  c_2\right) +\left(\dfrac{3 \gamma }{\beta }- \gamma  \varepsilon - 2 c_2\right) \left(1-\dfrac{3 \gamma ^2}{\beta } + \beta  \varepsilon + 2 \gamma  c_2\right). \hfill \\
\end{aligned}
\end{equation}

\smallskip

By setting: $\varepsilon = 1/10$, $\beta = 0.47$, $a = -2$, $b = 4$, $d=3$, $c_1 = 280/729$, $c_2 = -26/27$ and while considering that the ``duck parameter'' $\gamma$ can vary, $D_1$ and $D_2$ are respectively polynomial equations of degree two and three in $\gamma$. These quadratic and cubic functions $D_1$ and $D_2$ have been plotted on Fig. 5. One can see that between the lower limit called $\gamma_{saddle-node}$ and, the upper limit called $\gamma_{Hopf}$ corresponding to the value of the parameter $\gamma$ for which the real parts of both complex eigenvalues vanishes (see \textit{proof} in Appendix C), $D_1$ and $D_2$ are strictly positive. So, for $\gamma \in [\gamma_{saddle-node}, \gamma_{Hopf}]$ (purple rectangle on Fig. 5) the fixed points are stable while for $\gamma > \gamma_{Hopf}$ they are unstable. With this set of parameters,

\[
\gamma_{Hopf} \approx 0.274.
\]

\begin{figure}[htbp]
\centerline{\includegraphics[width = 11cm,height = 11cm]{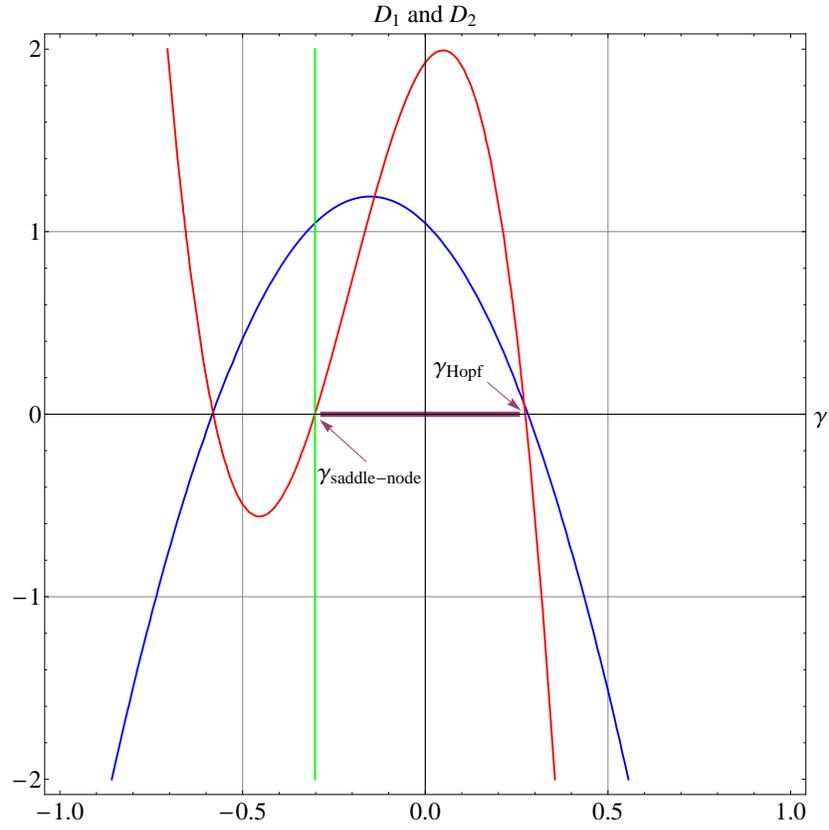}}
\caption{Routh-Hurwitz determinants of system (\ref{eq56}). $D_1$ in blue, $D_2$ in red and the saddle-node axis $\gamma = 2 \beta c_2/3$ in green for parameter values: $\varepsilon = 1/10$, $\beta = 0.47$, $a = -2$, $b = 4$, $d=3$, $c_1 = 280/729$ and $c_2 = -26/27$.}
\label{fig5}
\end{figure}

Thus, it appears from what precedes and from Prop. 1 that ``canards solutions'' may be observed in system (\ref{eq56}) for $\gamma_{Duck}$ values such that:

\begin{equation}
\label{eq66}
\gamma_{saddle-node} = \dfrac{  2 \beta c_2 }{3} < \gamma_{Hopf} < \gamma_{Duck}
\end{equation}

On Fig. 6, numerical ``canards solutions'' and \textit{slow manifold} of system (\ref{eq56}) have been plotted for the ``duck parameter'' $\gamma_{Duck} = 0.3275$ (all other parameters are the same as indicated above). Due to the symmetry of the system (\ref{eq56}), any of the two \textit{pseudo singular points} plotted in green on Fig. 6 was chosen as initial condition.

\begin{figure}[htbp]
\centerline{\includegraphics[width = 11cm,height = 11cm]{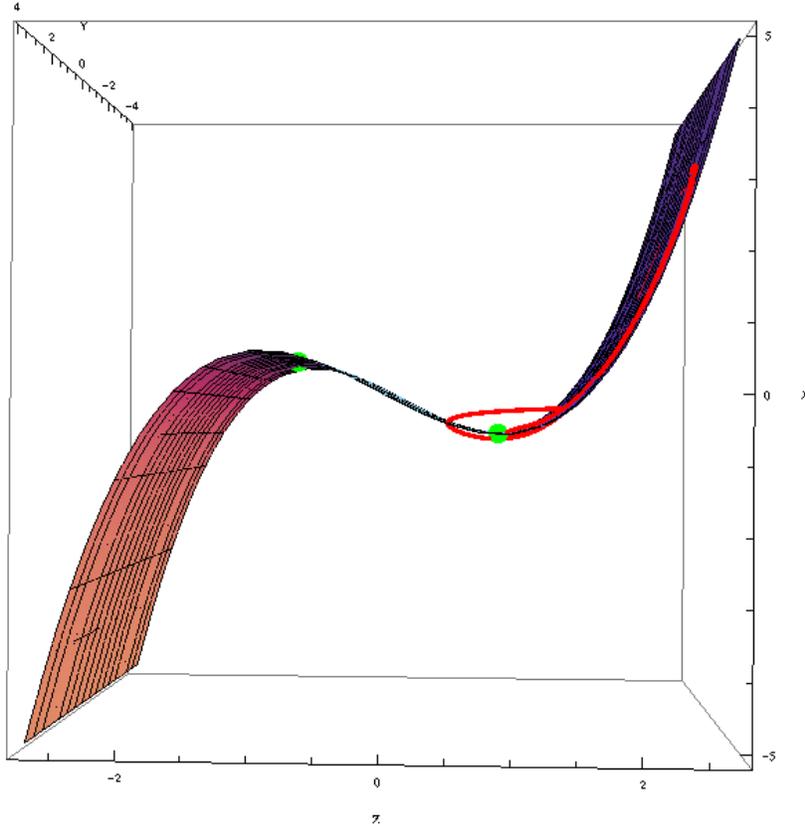}}
\caption{Numerical ``canards solutions'' and \textit{slow manifold} of system (\ref{eq56}) for parameter values: $\varepsilon = 1/10$, $\beta = 0.47$, $\gamma_{Duck} = 0.3275$, $a = -2$, $b = 4$, $d=3$, $c_1 = 280/729$ and $c_2 = -26/27$.}
\label{fig6}
\end{figure}

\subsection{Particular case}

In a previous work entitled ``Canards from Chua's circuits'', Ginoux \textit{et al.} \cite{GinouxLLibre2013} have studied the system (\ref{eq56}) with the following particular parameters:

\[
\gamma = \beta = \alpha \quad c_1 = \frac{1}{3} \quad c_2 = -1
\]

First, let's replace these parameters in the above conditions $C_1$ and $C_2$ (\ref{eq61}). We have:

\[
\begin{aligned}
C_1:& \quad \Delta = 1 + \dfrac{40 \alpha}{3} > 0, \hfill \\
C_2:& \quad q = - \dfrac{10 \alpha}{3}  < 0.
\end{aligned}
\]

Obviously, if $\alpha >0$, then both conditions $C_1$ and $C_2$ are verified. This is exactly the result provided by Itoh and Chua \cite{ItohChua1992} as it has been noticed in Ginoux \textit{et al.} \cite[p. 1330010-4]{GinouxLLibre2013}. However, it has been also remarked in our same previous paper \cite{GinouxLLibre2013} that the system (\ref{eq56}) admits, except the origin, two fixed points, the stability of which could preclude the existence of ``canards solutions''. By setting: $\gamma = \beta = \alpha$, $c_1 = \frac{1}{3}$ and $c_2 = -1$ in Eq. (\ref{eq63}) we find again the fixed points obtained by Ginoux \textit{et al.} \cite[p. 1330010-4]{GinouxLLibre2013}:

\[
x_1^* = \pm \sqrt{6}, \quad x_2^* = \mp \sqrt{6}, \quad y_1^* = \mp \sqrt{6}.
\]

Moreover, still using the Routh-Hurwitz' theorem and by setting: $\gamma = \beta = \alpha$, $c_1 = \frac{1}{3}$ and $c_2 = -1$ in Eq. (\ref{eq65}) we find that:

\[
\begin{aligned}
D_1 & = 1 - 5 \alpha + \varepsilon \alpha, \hfill \\
D_2 & = \alpha^2 \left(5 \varepsilon -\varepsilon ^2\right) - 25 \alpha + 5. \hfill \\
\end{aligned}
\]

\smallskip

Functions $D_1$ and $D_2$ have been plotted on Fig. 7 on which one can see that between the lower limit called $\alpha_{saddle-node}$ and, the upper limit called $\alpha_{Hopf}$ corresponding to the value of the parameter $\alpha$ for which the real parts of both complex eigenvalues vanishes, $D_1$ and $D_2$ are strictly positive. So, for $\alpha \in [\alpha_{saddle-node}, \alpha_{Hopf}]$ (purple rectangle on Fig. 7) the fixed points are stable while for $\alpha > \alpha_{Hopf}$, \textit{i.e.} $\alpha > 1/5$ they are unstable.

\begin{figure}[htbp]
\centerline{\includegraphics[width = 11cm,height = 11cm]{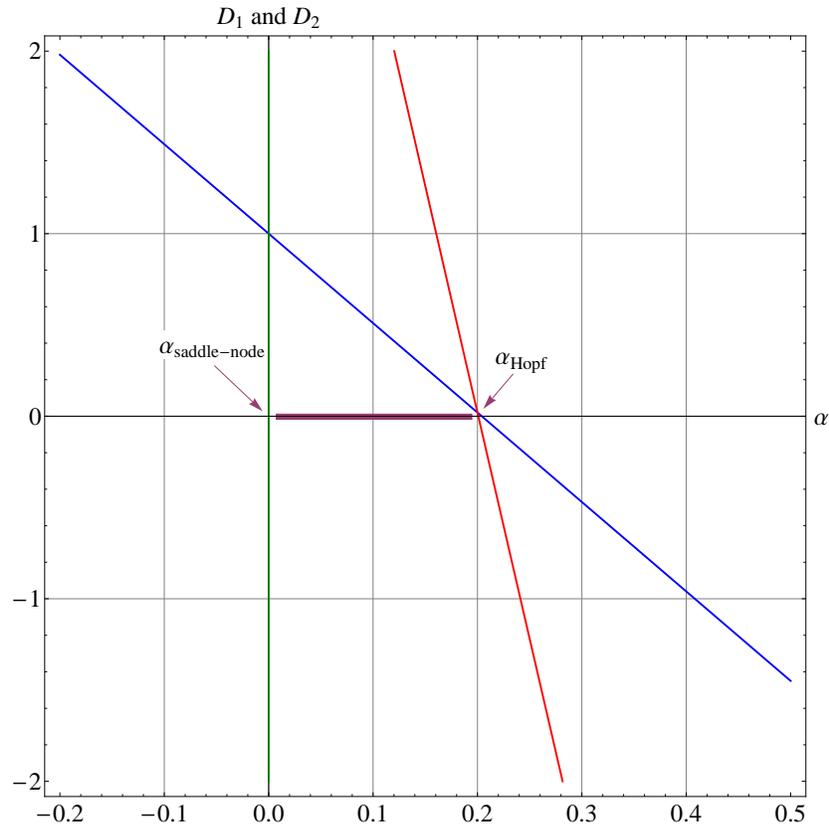}}
\caption{Routh-Hurwitz determinants of system (\ref{eq56}). $D_1$ in blue, $D_2$ in red and the the saddle-node axis $\gamma = 0$ in green for parameter values: $\varepsilon = 1/10$, $\beta = \gamma = \alpha$, $c_1 = 1/3$ and $c_2 = -1$.}
\label{fig7}
\end{figure}

Thus, it appears from what precedes and from Prop. \ref{prop2} that ``canards solutions'' may be observed in system (\ref{eq56}) provided that:

\[
\alpha_{saddle-node} = 0 < \alpha_{Hopf} < \alpha_{Duck}
\]

This is exactly the result obtained by Ginoux \textit{et al.} \cite[p. 1330010-6]{GinouxLLibre2013}. The phase portrait of system (\ref{eq56}) with this set of parameter values has already been published by Ginoux \textit{et al.} \cite{GinouxLLibre2013}.

\section{Fourth-order Memristor-Based canonical oscillator}
\label{Sec6}

Let's consider again the Memristor-Based canonical Chua's circuit \cite{ItohChua2008,ItohChua2013}. By adding an inductor in parallel with conductance $-G$, Fitch \textit{et al.} \cite{Fitch2012} have modified this circuit in order to obtain a \textit{fourth-order Memristor-Based canonical oscillator} (see Fig. 8).

\begin{figure}[htbp]
\centerline{\includegraphics[width = 8.5cm,height = 3.145cm]{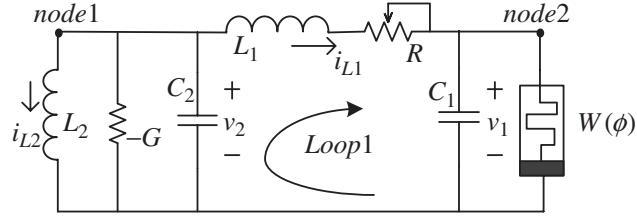}}
\caption{Memristor canonical Chua's circuit \cite{Fitch2012}.}
\label{Fig8}
\end{figure}

\subsection{Flux-linkage and charge phase space}

Applying Kirchhoff's circuit laws to the nodes $1$, $2$ and the loop $1$ of the circuit Fig. 8, Fitch \textit{et al.} \cite{Fitch2012} obtained the following set of differential equations, \textit{i.e.}, the following \textit{memristor based chaotic circuit}:

\begin{equation}
\label{eq67}
\begin{aligned}
C_1 \dfrac{d\varphi_1}{dt} & = R q_1 - k\left( \varphi_1 \right), \hfill \\
C_2 \dfrac{d\varphi_2}{dt} & = - q_2 + G \varphi_2 - q_1, \hfill \\
L_1 \dfrac{dq_1}{dt} & = \varphi_2 - \varphi_1 - R q_1, \hfill \\
L_2 \dfrac{dq_3}{dt} & = \varphi_2,
\end{aligned}
\end{equation}

\vspace{0.1in}

where the classical piecewise-linear function $k(\varphi)$ of the Chua's memristor (\ref{eq50}) has been replaced by the cubic $\hat{k}(\varphi) = c_1 \varphi^3 + c_2 \varphi$.

\vspace{0.1in}

By setting $x= \varphi_1$, $y = q_1$, $z = \varphi_2$, $u=q_2$, $C_1 = \varepsilon$, $C_2 = 1$, $\beta_1 = \dfrac{1}{L_1}$, $\beta_2 = \dfrac{1}{L_2}$, $G = -\alpha_2$ and $R =1$ the \textit{memristor based chaotic circuit} (\ref{eq67}) can be written:

\begin{equation}
\label{eq68}
\begin{aligned}
\varepsilon \dfrac{dx}{dt} & = y - k\left( x \right), \hfill \\
\dfrac{dy}{dt} & = - u - \alpha_2 z - y, \hfill \\
\dfrac{dz}{dt} & = \beta_1 \left(z - x - y \right), \hfill \\
\dfrac{du}{dt} & = \beta_2 z.
\end{aligned}
\end{equation}

\vspace{0.1in}

Now, let's make the following variable changes in Eqs. (\ref{eq68}) in order to apply the method presented in Sec. 4:

\[
x \rightarrow u, \quad y \rightarrow x, \quad u \rightarrow y.
\]

Thus, we have:

\begin{equation}
\label{eq69}
\begin{aligned}
\dfrac{dx}{dt} & = \beta_1 \left(z - x - u \right), \hfill \\
\dfrac{dy}{dt} & = \beta_2 z, \hfill \\
\dfrac{dz}{dt} & = - y - \alpha_2 z - x, \hfill \\
\varepsilon \dfrac{du}{dt} & = x - k\left( u \right).
\end{aligned}
\end{equation}

\vspace{0.1in}

Let's notice that system (\ref{eq69}) is exactly identical to that studied by Ginoux \textit{et al.} \cite{GinouxLLibre2013}. Thus, condition (we will provide below) for the existence of canard solutions in system (\ref{eq69}) will be compared to that given in our previous works entitled ``Canards from Chua's circuits'' \cite{GinouxLLibre2013}.\\

Finally, let's replace the variables ($x,y,z,u$) by ($x_1, x_2, x_3, y_1$) and let's apply the method presented in Sec. 4 to the following system (\ref{eq56}) where $k(y_1) = c_1 y_1^3 + c_2 y_1$.

\begin{equation}
\label{eq70}
\begin{aligned}
\dfrac{dx_1}{dt} & = \beta_1 \left(x_3 - x_1 - y_1 \right), \hfill \\
\dfrac{dx_2}{dt} & = \beta_2 x_3, \hfill \\
\dfrac{dx_3}{dt} & = - x_2 - \alpha_2 x_3 - x_1, \hfill \\
\varepsilon \dfrac{dy_1}{dt} & = x_1 - k\left( y_1 \right).
\end{aligned}
\end{equation}

\subsection{Critical manifold and contrained system}

The critical manifold of this system (\ref{eq70}) is given by $x_1 - k(y_1) = 0$. According to Eq. (\ref{eq35}) the constrained system on the critical manifold reads:

\begin{equation}
\label{eq71}
\begin{aligned}
\dfrac{dx_1}{dt} & = \beta_1 \left(x_3 - x_1 - y_1 \right), \hfill \\
\dfrac{dx_2}{dt} & = \beta_2 x_3, \hfill \\
\dfrac{dx_3}{dt} & = - x_2 - \alpha_2 x_3 - x_1, \hfill \\
\dfrac{dy_1}{dt} & = - \frac{ \beta_1 \left( x_3 - x_1 - y_1 \right) }{ - \left( 3 c_1 y_1^2 + c_2  \right) }, \hfill \vspace{6pt} \\
0 & = x_1 - k\left( y_1 \right).
\end{aligned}
\end{equation}

\subsection{Normalized slow dynamics}

Then, by rescaling the time by setting $t = -  \dfrac{\partial g_1}{\partial y_1} \tau = (3 c_1 y_1^2 + c_2) $ we obtain the ``normalized slow dynamics'':

\begin{equation}
\label{eq72}
\begin{aligned}
\dfrac{dx_1}{dt} & = \beta_1 \left(x_3 - x_1 - y_1 \right)\left( 3 c_1 y_1^2 + c_2  \right) = F_1 \left( x_1, x_2, x_3, y_1 \right), \hfill \\
\dfrac{dx_2}{dt} & = \beta_2 x_3 \left( 3 c_1 y_1^2 + c_2  \right) = F_2 \left( x_1, x_2, x_3, y_1 \right), \hfill \\
\dfrac{dx_3}{dt} & = \left( - x_2 - \alpha_2 x_3 - x_1 \right)\left( 3 c_1 y_1^2 + c_2  \right) = F_3 \left( x_1, x_2, x_3, y_1 \right), \hfill \\
\dfrac{dy_1}{dt} & = \beta_1 \left( x_3 - x_1 - y_1 \right) = G_1 \left( x_1, x_2, x_3, y_1 \right), \hfill \vspace{6pt} \\
0 & = x_1 - k\left( y_1 \right).
\end{aligned}
\end{equation}

\subsection{Pseudo singular manifold}

According to Eqs. (\ref{eq38}), the \textit{pseudo singular manifold} of system (\ref{eq70}) is defined by:

\begin{equation}
\label{eq73}
\left( \tilde{x}_{1},x_2, \tilde{x}_{3}, \tilde{y}_{1} \right) = \left(\pm \dfrac{2c_2}{3}\sqrt{\dfrac{- c_2}{3c_1}}, x_2,
\pm ( \dfrac{2c_2}{3} + 1 ) \sqrt{\dfrac{- c_2}{3c_1}}, \pm \sqrt{\dfrac{- c_2}{3c_1}} \right)
\end{equation}

\vspace{0.1in}

Let's notice that $\tilde{x}_{2}$ is undetermined. In ``Canards from Chua's circuit'', Ginoux \textit{et al.} \cite{GinouxLLibre2013} have arbitrarily chosen $\tilde{x}_{2} =0$. We will see in the following that this choice does not affect their results.\\

The Jacobian matrix of system (\ref{eq72}) evaluated at ($\tilde{x}_1^{\pm}, x_2, \tilde{x}_3^{\pm}, \tilde{y}_1^{\pm}$) reads:

\begin{equation}
\label{eq74}
J_{(F_1, F_2, F_3, G_1)} =  \begin{pmatrix}
0  &  0  &  0  &  0 \vspace{6pt} \\
0  &  0  &  0  &  6 \beta_2 c_1 \tilde{x}_{3\pm} \tilde{y}_{1\pm} \vspace{6pt} \\
0  &  0  &  0  &  -6 c_1 \left( \alpha_2 \tilde{x}_{3\pm} + x_2 + c_1 \tilde{y}_{1\pm}^3 + c_2 \tilde{y}_{1\pm} \right)\tilde{y}_{1\pm}  \vspace{6pt} \\
- \beta_1  &  0  &   \beta_1  &  - \beta_1  \vspace{6pt}
\end{pmatrix}
\end{equation}

\begin{remark}
Although, the \textit{pseudo singular manifold} has not been transformed into ($0, x_2, 0, 0$) extension of Beno\^{i}t's generic hypotheses (\ref{eq40}-\ref{eq41}) are satisfied.
\end{remark}

\subsection{Canard existence in fourth-order memristor Chua's circuit}

According to Eqs. (\ref{eq46}) we find that:

\begin{equation}
\label{eq75}
\begin{aligned}
p & = Tr(J) = -\beta_1, \hfill  \\
q & = \sigma_2 = + 6 \beta_1 c_1 \left( \alpha_2 \tilde{x}_{3\pm} + x_2 + \tilde{x}_{1\pm} \right)\tilde{y}_{1\pm}
\end{aligned}
\end{equation}

\smallskip

Thus, the conditions $C_1$ and $C_2$ for ($\tilde{x}_1^{\pm}, x_2, \tilde{x}_3^{\pm}, \tilde{y}_1^{\pm}$) to be of saddle type reads:

\begin{equation}
\label{eq76}
\begin{aligned}
C_1:& \quad \Delta = \beta_1 \left[ \beta_1 - 24 c_1 \left( \alpha_2 \tilde{x}_{3\pm} + x_2 + \tilde{x}_{1\pm} \right)\tilde{y}_{1\pm}  \right] > 0, \hfill \\
C_2:& \quad q= + 6 \beta_1 c_1 \left( \alpha_2 \tilde{x}_{3\pm} + x_2 + \tilde{x}_{1\pm} \right)\tilde{y}_{1\pm}  < 0.
\end{aligned}
\end{equation}

\smallskip

Then, due to the nature ($\pm$) of the \textit{pseudo singular manifold} (\ref{eq73}) we have two cases corresponding to the positive and negative values.

\subsection{Positive case}

Let's consider the positive case for which the \textit{pseudo singular manifold} (\ref{eq73}) can be written as:

\[
 \left( \tilde{x}_1^{+}, x_2, \tilde{x}_3^{+}, \tilde{y}_1^{+} \right) = \left( + \dfrac{2c_2}{3}\sqrt{\dfrac{- c_2}{3c_1}}, x_2,
+ ( \dfrac{2c_2}{3} + 1 ) \sqrt{\dfrac{- c_2}{3c_1}}, + \sqrt{\dfrac{- c_2}{3c_1}} \right).
\]

\smallskip

Conditions $C_1$ and $C_2$ reads then:

\begin{subequations}
\label{eq77}
\begin{align}
C_1:& \quad \alpha_2 \tilde{x}_{3+} + x_2 + \tilde{x}_{1+} < \dfrac{\beta_1}{24 c_1 \tilde{y}_{1+}} , \hfill \\
C_2:& \quad \alpha_2 \tilde{x}_{3+} + x_2 + \tilde{x}_{1+}  < 0.
\end{align}
\end{subequations}

\smallskip

Obviously, since the right hand side of the first inequality (77a) is positive ($\beta_1 >0$, $c_1 >0$ and $\tilde{y}_{1+} >0$), both conditions are satisfied provided that the condition $C_2$ is verified. So, to have \textit{pseudo singular manifold} of saddle type, the straight line $\alpha_2 \tilde{x}_{3+} + x_2 + \tilde{x}_{1+}$ must verify:

\begin{equation}
\label{eq78}
\alpha_2 \tilde{x}_{3+} + x_2 + \tilde{x}_{1+} < 0.
\end{equation}

\smallskip

By taking into account the above preliminary result and while fixing all the parameters excepted $\alpha_2$, this straight line ($D_{+}$) can be plotted in the plane ($x_2, \alpha_2$) and reads:

\begin{equation}
\label{eq79}
(D_{+}): \quad  \alpha_2 \left( \dfrac{2c_2}{3} + 1 \right) \sqrt{\dfrac{- c_2}{3c_1}} + x_2 + \dfrac{2c_2}{3} \sqrt{\dfrac{- c_2}{3c_1}} < 0.
\end{equation}

Let's notice that for:

\[
\begin{aligned}
\alpha_2 & = 0, \quad x_2 = - \dfrac{2c_2}{3} \sqrt{\dfrac{- c_2}{3c_1}}, \hfill \\
x_2 & = 0, \quad \alpha_2 = - \dfrac{2c_2}{2c_2 + 3}. \hfill
\end{aligned}
\]

\newpage

\subsection{Negative case}

Let's consider the negative case for which the \textit{pseudo singular manifold} (\ref{eq73}) can be written as:

\[
\left( \tilde{x}_1^{-}, x_2, \tilde{x}_3^{-}, \tilde{y}_1^{-} \right) = \left( - \dfrac{2c_2}{3}\sqrt{\dfrac{- c_2}{3c_1}},
- ( \dfrac{2c_2}{3} + 1 ) \sqrt{\dfrac{- c_2}{3c_1}},
- \sqrt{\dfrac{- c_2}{3c_1}} \right).
\]

Conditions $C_1$ and $C_2$ reads then:

\begin{subequations}
\label{eq80}
\begin{align}
C_1:& \quad \dfrac{\beta_1}{24 c_1 \tilde{y}_{1-}}  < \alpha_2 \tilde{x}_{3-} + x_2 + \tilde{x}_{1-}, \hfill \\
C_2:& \quad 0 < \alpha_2 \tilde{x}_{3-} + x_2 + \tilde{x}_{1-}.
\end{align}
\end{subequations}

\smallskip

Obviously, since the left hand side of the first inequality (80a) is negative ($\beta_1 >0$, $c_1 >0$ and $\tilde{y}_{1-} < 0$), both conditions are satisfied provided that the condition $C_2$ is satisfied. So, to have \textit{pseudo singular manifold} of saddle type, the straight line $\alpha_2 \tilde{x}_{3-} + x_2 + \tilde{x}_{1-}$ must verify:

\begin{equation}
\label{eq81}
\alpha_2 \tilde{x}_{3-} + x_2 + c_1 \tilde{y}_{1-}^3 + c_2 \tilde{y}_{1-} > 0.
\end{equation}

\smallskip

By taking into account the above preliminary result and while fixing all the parameters excepted $\alpha_2$, this straight line ($D_{-}$) can be plotted in the plane ($x_2, \alpha_2$) and reads:

\begin{equation}
\label{eq82}
(D_{-}): \quad  - \alpha_2 \left( \dfrac{2c_2}{3} + 1 \right) \sqrt{\dfrac{- c_2}{3c_1}} + x_2 - \dfrac{2c_2}{3} \sqrt{\dfrac{- c_2}{3c_1}} < 0.
\end{equation}

Let's notice that for:

\[
\begin{aligned}
\alpha_2 & = 0, \quad x_2 = \dfrac{2c_2}{3} \sqrt{\dfrac{- c_2}{3c_1}}, \hfill \\
x_2 & = 0, \quad \alpha_2 = - \dfrac{2c_2}{2c_2 + 3}. \hfill
\end{aligned}
\]

\smallskip

\begin{figure}[htbp]
  \begin{center}
    \begin{tabular}{c}
      \includegraphics[width=11cm,height=11cm]{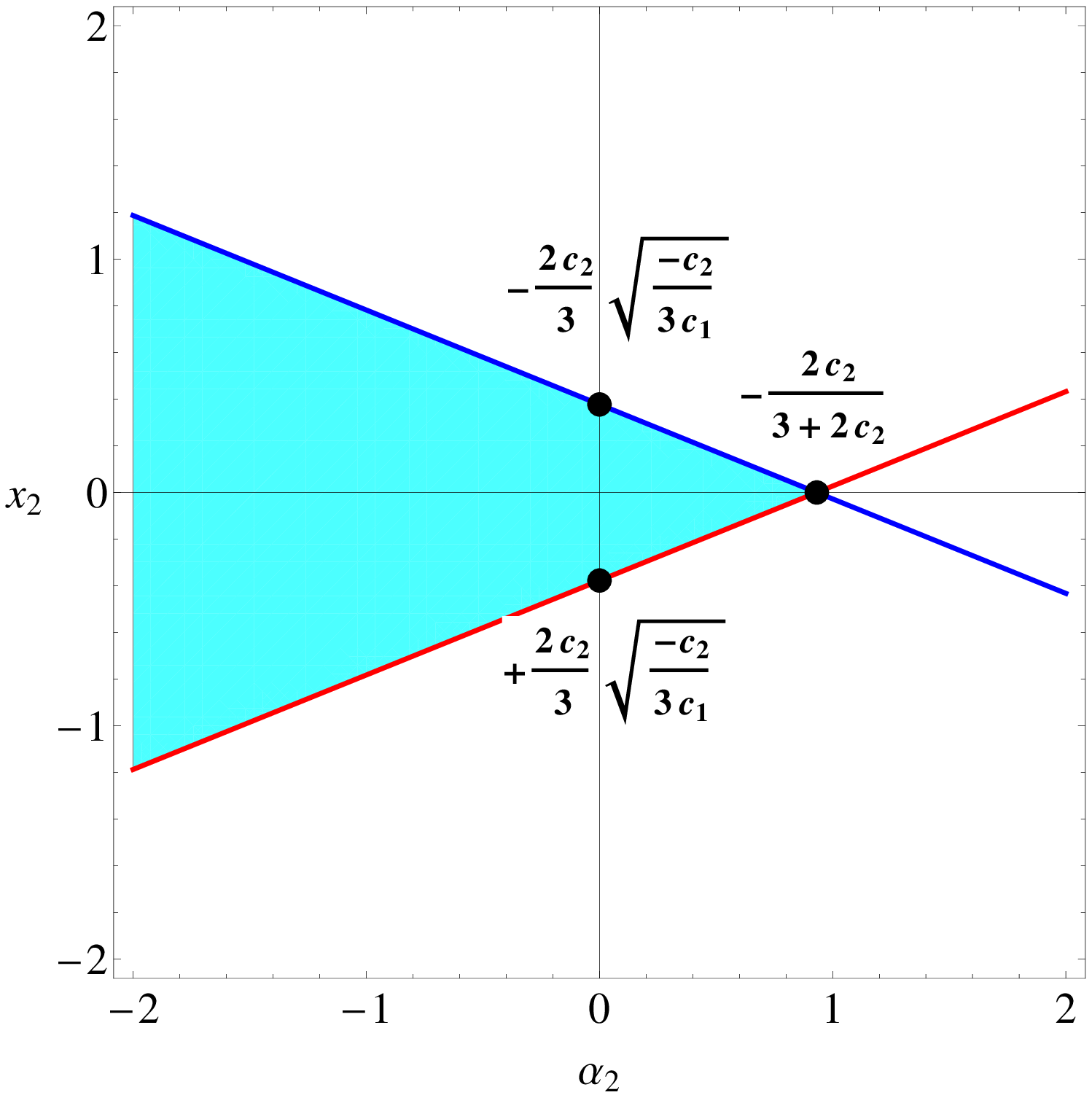} \\
     \hspace{1.75cm} (a) \hspace{0.85cm} \\[0.2cm]
      \includegraphics[width=11cm,height=11cm]{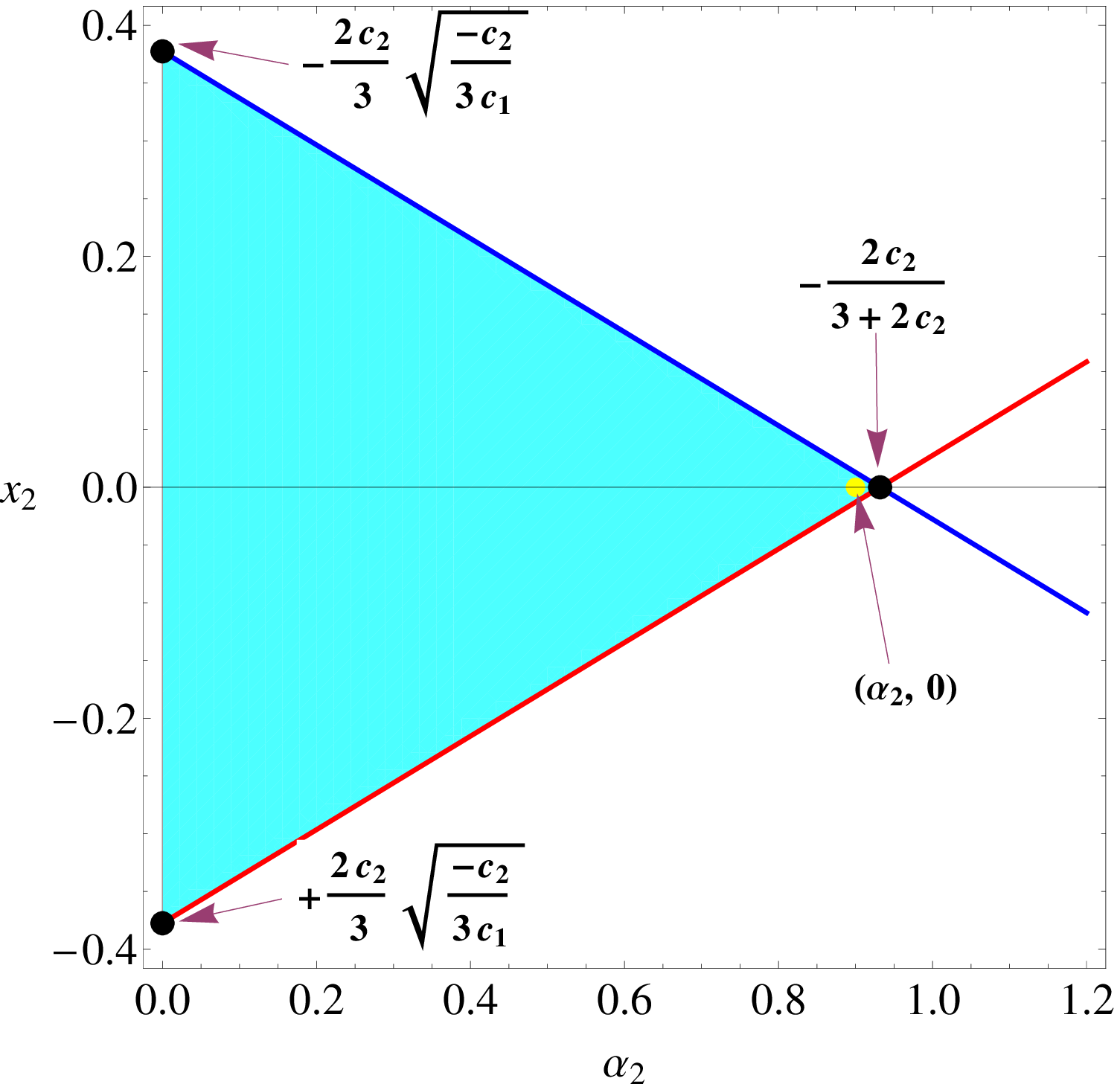} \\
     \hspace{1.75cm} (b) \hspace{0.6cm} \\[0.2cm]
    \end{tabular}
    \caption{Region within which the \textit{pseudo singular points} of fourth-order memristor Chua's circuit (\ref{eq70}) are of saddle type.}
    \label{fig9}
  \end{center}
  \vspace{-0.5cm}
\end{figure}

\smallskip

On Fig. 9a \& 9b, the straight lines ($D_{+}$) and ($D_{-}$) have been plotted in blue and in red (respectively). Thus, the region within which the \textit{pseudo singular points} are of saddle type corresponds to the cyan triangle. A zoom of Fig. 9a is presented on Fig. 9b. Let's notice on the one hand that the point ($\alpha_2 = 0.9, x_2 =0$) arbitrarily chosen by Ginoux \textit{et al.} \cite{GinouxLLibre2013} and plotted in yellow on Fig. 9b belongs to the cyan region within which the \textit{pseudo singular points} are of saddle type. On the other hand, this cyan triangular region is limited on the right, at the top of the triangle, by the point of coordinate ($\alpha_2 = - \dfrac{2c_2}{2c_2 + 3}, x_2=0$) which corresponds exactly with the condition stated in Ginoux \textit{et al.} \cite{GinouxLLibre2013} and for which canard solutions have been observed in Chua's system 4D (\ref{eq70}) according to Prop. 2. In other words, to have a \textit{pseudo singular point} of saddle type at $x_2 = 0$, $\alpha_2 < - \frac{2c_2}{2c_2 + 3}$. To confirm this fact, the two nonzero eigenvalues of the characteristic polynomial associated with the Jacobian matrix (\ref{eq74}) evaluated at ($\tilde{x}_{1},x_2, \tilde{x}_{3}, \tilde{y}_{1}$) (\ref{eq73}) have been computed for $\alpha_2 = 0.9$ and for the corresponding values of $x_2$ which have been taken equal to zero by Ginoux \textit{et al.} \cite{GinouxLLibre2013} but which is in fact very small $x_2 = \mp 0.01$. We have found that the two nonzero real eigenvalues are of opposite sign what corresponds to the case of \textit{pseudo singular points} of saddle type.\\

So, the value of the ``duck parameter'' $\alpha_2$ for which the \textit{pseudo singular points} are of saddle-type is defined by:

\begin{equation}
\label{eq83}
\alpha_2 < \alpha_{2saddle-node} = - \dfrac{2c_2}{3 + 2c_2}.
\end{equation}

\smallskip

where $\alpha_{2saddle-node}$ represents the critical value of the parameter $\alpha_2$ for which one of the two remaining eigenvalues $\lambda_1$ or $\lambda_2$ of the eigenpoynomial associated with the Jacobian matrix (\ref{eq74}) vanishes. With this set of parameters $\varepsilon = 1/10.1428$, $\beta _1 = 0.121$, $\beta _2 = 0.0047$, $c_1 = 0.393781$ and $c_2 = -0.72357$,

\[
\alpha_{2saddle-node} = - \dfrac{2c_2}{3 + 2c_2} \approx 0.932.
\]

\subsection{Fixed points stability and Routh-Hurwitz' theorem}

However, as pointed out in the previous Sect. 5.7 the system (\ref{eq70}) admits the origin $O(0,0,0,0)$ as fixed point, the stability of which could preclude the existence of ``canards solutions''. The \textit{eigenpolynomial} equation of the Jacobian matrix of system (\ref{eq70}) evaluated at this fixed point reads:

\begin{equation}
\label{eq84}
a_4\lambda^4 +a_3 \lambda^3 + a_2 \lambda^2  + a_1 \lambda  + a_0 = 0
\end{equation}

where

\[
\begin{aligned}
a_0 & = \left(1 + c_2\right)\beta_1 \beta_2, \hfill \\
a_1 & = c_2 \left(\left( 1 + \alpha_2 \right) \beta_1 + \beta _2 \right) + \beta_1 \left(\alpha_2 + \varepsilon \beta_2\right), \hfill \\
a_2 & = \left(1 + \varepsilon + \varepsilon \alpha_2 \right) \beta_1 + c_2 \left(\alpha_2 + \beta_1 \right) + \varepsilon \beta_2, \hfill \\
a_3 & = c_2 + \varepsilon \left(\alpha_2 + \beta_1 \right), \hfill \\
a_4 & = \varepsilon.
\end{aligned}
\]

Let suppose that all the parameters are fixed except $\alpha_2$, \textit{i.e.} the ``duck parameter'' and, let's make use again of the Routh-Hurwitz' theorem \cite{Routh1877,Hurwitz1893}. Thus, it states that if $D_1 = a_1$, $D_2 = a_1a_2 - a_0 a_3$ and $D_3 = a_1 a_2 a_3 - a_0 a_3^2 - a_1^2 a_4$ are all positive then \textit{eigenpolynomial} equation would have eigenvalues with real negative parts. In other words, if $D_1$, $D_2$ and $D_3$ are positive the fixed point will be stable.

\begin{figure}[htbp]
\centerline{\includegraphics[width = 11cm,height = 11cm]{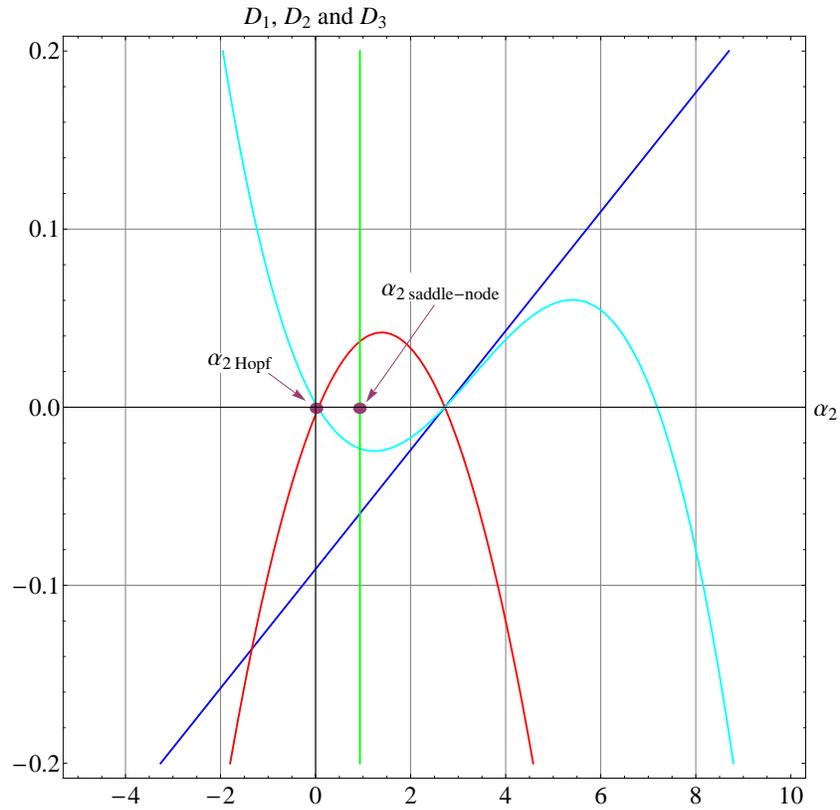}}
\caption{Routh-Hurwitz determinants of system (\ref{eq70}). $D_1$ in blue, $D_2$ in red, $D_3$ in cyan and the saddle-node axis $\alpha_{2saddle-node} = -2c_2 / (3 + 2c_2)$ in green for parameter values: $\varepsilon = 1/10.1428$, $\beta _1 = 0.121$, $\beta _2 = 0.0047$, $c_1 = 0.393781$ and $c_2 = -0.72357$.}
\label{fig10}
\end{figure}

By setting $\varepsilon = 1/10.1428$, $\beta _1 = 0.121$, $\beta _2 = 0.0047$, $c_1 = 0.393781$ and $c_2 = -0.72357$ and while considering that the ``duck parameter'' $\alpha_2$ can vary, the functions $D_1$, $D_2$ and $D_3$ have been plotted on Fig. 10. One can see that between the lower limit called $\alpha_{2Hopf}$ corresponding to the value of the parameter $\alpha_2$ for which the real parts of both complex eigenvalues vanishes (see \textit{Proof} in the Appendix D.) and the upper limit called $\alpha_{2saddle-node}$, $D_1$ and $D_3$ are negative while $D_2$ is positive. So, in this interval, the fixed point is unstable. With this set of parameters,

\[
\alpha_{2Hopf} \approx 0.0451 \quad \mbox{ and } \quad \alpha_{2saddle-node} = - \dfrac{2c_2}{3 + 2c_2} \approx 0.932.
\]

Thus, we deduce from what precedes and from Prop. 2 that ``canards solutions'' may be observed in system (\ref{eq70}) provided that:

\begin{equation}
\label{eq85}
\alpha_{2Hopf} < \alpha_{2Duck} < \alpha_{2saddle-node} = \dfrac{-2c_2}{3 + 2c_2}
\end{equation}

\smallskip

On Figs. 11 \& 12, numerical ``canards solutions'' and \textit{critical manifold} of system (\ref{eq70}) have been plotted for the ``duck parameter'' $\alpha_{2Duck} = 0.1$ (all other parameters are the same as indicated above). Due to the symmetry of the system (\ref{eq70}), any of the two \textit{pseudo singular points} plotted in green on Figs. 11 \& 12 was chosen as initial condition.

\begin{figure}[htbp]
\centerline{\includegraphics[width = 11cm,height = 11cm]{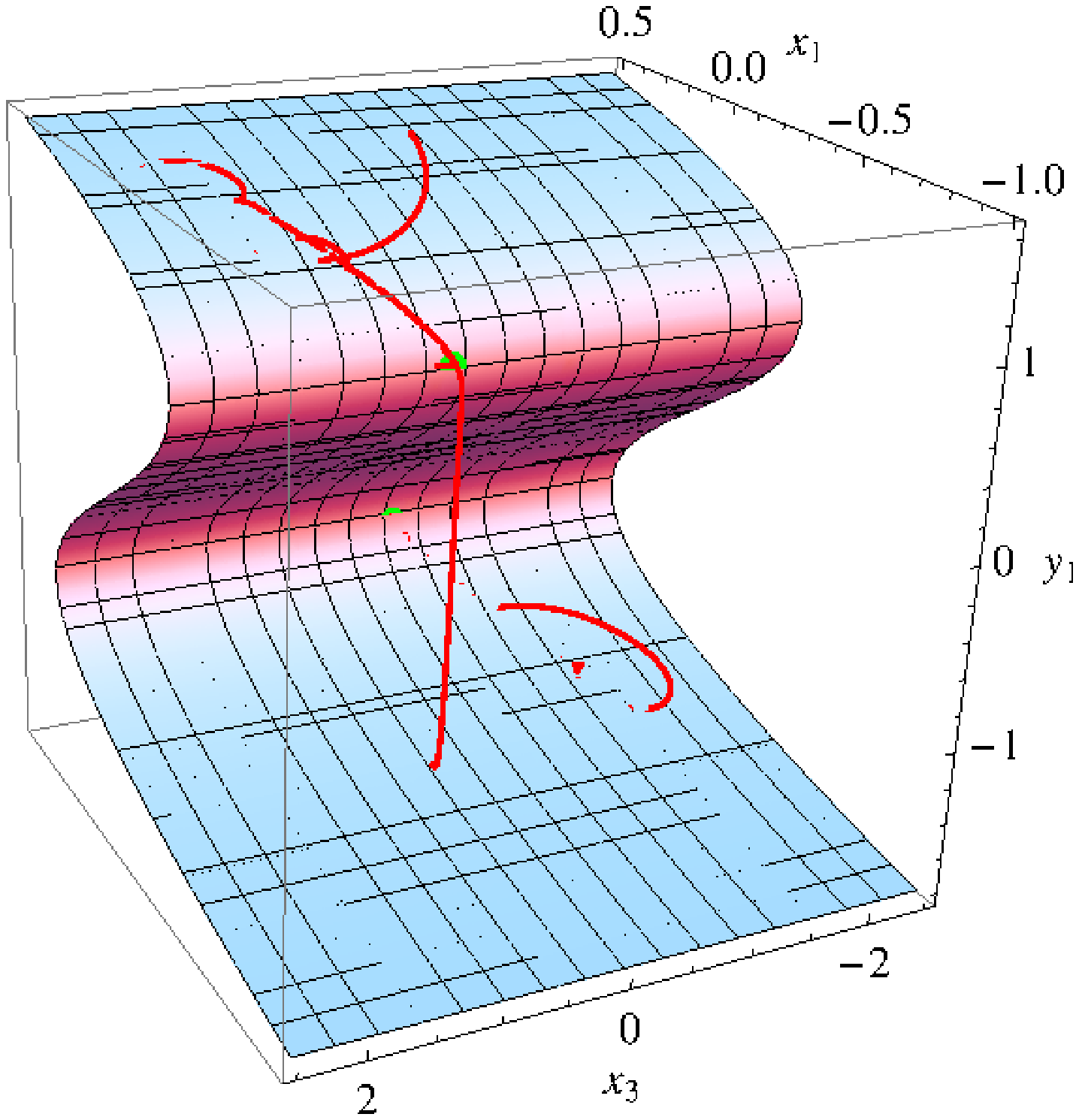}}
\vspace{0.1in}
\caption{Numerical ``canards solutions'' and \textit{critical manifold} of system (\ref{eq70}) in the ($x_1,x_3,y_1$) phase space for parameter values: $\varepsilon = 1/10.1428$, $\alpha_2 = 0.1$, $\beta _1 = 0.121$, $\beta _2 = 0.0047$, $c_1 = 0.393781$ and $c_2 = -0.72357$.}
\label{fig11}
\end{figure}

\begin{figure}[htbp]
\centerline{\includegraphics[width = 11cm,height = 11cm]{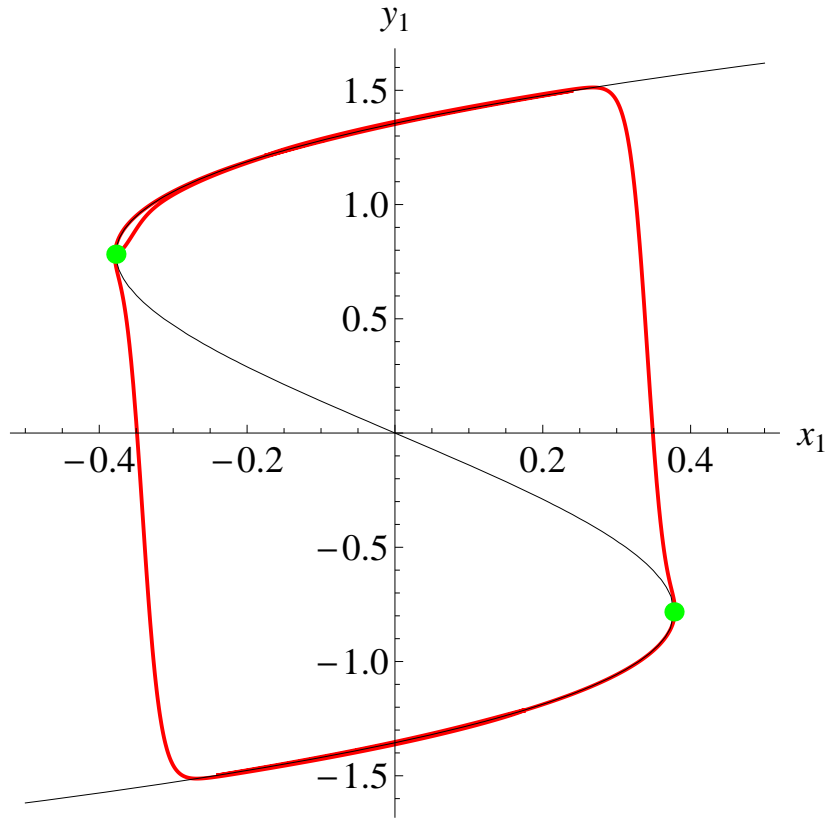}}
\vspace{0.1in}
\caption{Numerical ``canards solutions'' and \textit{critical manifold} of system (\ref{eq70}) ($x_1,y_1$) phase plane for parameter values: $\varepsilon = 1/10.1428$, $\alpha_2 = 0.1$, $\beta _1 = 0.121$, $\beta _2 = 0.0047$, $c_1 = 0.393781$ and $c_2 = -0.72357$.}
\label{fig12}
\end{figure}

\smallskip

\section{Discussion}

In this work we have proposed an alternative method for determining the condition of existence of ``canard solutions'' for three and four-dimensional singularly perturbed systems with only one \textit{fast} variable in the \textit{folded saddle} case. This method enables to highlight a unique generic condition ($\sigma_2 < 0$) for the existence of ``canard solutions'' for such three and four-dimensional singularly perturbed systems which is based on the stability of \textit{folded singularities} of the \textit{normalized slow dynamics} deduced from a well-known property of linear algebra. It has been stated that this unique generic condition was perfectly identical to that provided by Beno\^{i}t \cite{Benoit1983} and then by Szmolyan and Wechselberger \cite{SzmolyanWechselberger2001} and finally by Wechselberger \cite{Wechselberger2012}. Finally, it has been established that this condition is ``generic'' since it is exactly the same for singularly perturbed systems of dimension three and four with only one \textit{fast} variable. Application of this method to the famous three and four-dimensional memristor canonical Chua's circuits for which the classical piecewise-linear characteristic curve has been replaced by a smooth cubic nonlinear function according to the \textit{least squares method} has enabled to show the existence of ``canards solutions'' in such Memristor Based Chaotic Circuits.\\

However, in this paper, only the case of \textit{pseudo singular points} of saddle-type has been analyzed. Of course, the case of \textit{pseudo singular points} of node-type could be also studied with the same method. Moreover, this method could be successfully used for proving the existence of ``canard solutions'' in four-dimensional singularly perturbed systems with two \textit{fast} variables such as the famous Hodgkin-Huxley model or in the so-called coupled FitzHugh-Nagumo system. In a future work we will state that the existence of canard solutions in such systems can be established according to the same unique generic condition ($\sigma_2 < 0$). \\

\section{Acknowledgements}
We would like to thank to Ernesto P\'erez Chavela  for previous
discussions related with this work. The authors are partially supported by a MINECO/FEDER grant number
MTM2008-03437. The second author is partially supported by a
MICINN/FEDER grants numbers MTM2009-03437 and MTM2013-40998-P, by an
AGAUR grant number 2014SGR-568, by an ICREA Academia, two
FP7+PEOPLE+2012+IRSES numbers 316338 and 318999, and
FEDER-UNAB10-4E-378.

\renewcommand{\theequation}{A-\arabic{equation}}
\setcounter{equation}{0}  

\newpage

\section*{Appendix}

Change of coordinates leading to the \textit{normal forms} of three and four-dimensional singularly perturbed systems with one fast variable are given in the following section.

\subsection*{A. Normal form of 3D singularly perturbed systems with one fast variable}

Let's consider the three-dimensional \textit{singularly perturbed dynamical system} (\ref{eq11}) with $k=2$ \textit{slow} variables and $m=1$ \textit{fast} and let's make the following change of variables:

\begin{equation}
x_1 = \alpha^2 x, \quad  x_2 = \alpha y, \quad y_1 = \alpha z \quad \mbox{where} \quad \alpha << 1.
\end{equation}

By taking into account Beno\^{i}t's generic hypothesis Eqs. (\ref{eq20},\ref{eq21}) and while using Taylor series expansion the system (\ref{eq11}) becomes:

\begin{equation}
\label{eqA2}
\begin{aligned}
\dot{x} & = \dfrac{\partial f_1}{\partial y} y + \dfrac{\partial f_1}{\partial z} z, \hfill \vspace{6pt} \\
\dot{y} & = f_2 \left( x, y, z \right), \hfill \vspace{6pt} \\
\dfrac{\varepsilon}{\alpha^2} \dot{z} & = \dfrac{\partial g_1}{\partial x} x + \dfrac{1}{2} \dfrac{\partial^2 g_1}{\partial y^2} y^2 + \dfrac{\partial^2 g_1}{\partial y \partial z} y z + \dfrac{1}{2} \dfrac{\partial^2 g_1}{\partial z^2} z^2. \hfill
\end{aligned}
\end{equation}

\smallskip

Then, let's make the standard polynomial change of variables:

\begin{equation}
\label{eqA3}
\begin{aligned}
X & = A x + B y^2, \hfill \vspace{6pt} \\
Y & = \dfrac{y}{f_2}, \hfill \vspace{6pt} \\
Z & = Cy + Dz. \hfill
\end{aligned}
\end{equation}

\smallskip

From (A-3) we deduce that:

\begin{equation}
\label{eqA4}
\begin{aligned}
x & = \frac{X - B f_2^2 Y^2}{A}, \hfill \vspace{6pt} \\
y & = f_2 Y, \hfill \vspace{6pt} \\
z & = \frac{Z - C f_2 Y}{D}. \hfill
\end{aligned}
\end{equation}

The time derivative of system (A-3) gives:

\begin{equation}
\label{eqA5}
\begin{aligned}
\dot{X} & = A \dot{x} + 2 B v \dot{y}, \hfill \vspace{6pt} \\
\dot{Y} & = \dfrac{\dot{y}}{f_2}, \hfill \vspace{6pt} \\
\dot{Z} & =  C\dot{y} + D\dot{z}. \hfill
\end{aligned}
\end{equation}

\smallskip

Then, multiplying the third equation of (A-5) by $( \varepsilon / \alpha^2)$ and while replacing in (A-5) $\dot{x}$, $\dot{y}$ and $\dot{z}$ by the right-hand-side of system (A-2) leads to:

\begin{equation}
\label{eqA6}
\begin{aligned}
\dot{X} & = A \left( \dfrac{\partial f_1}{\partial y} y + \dfrac{\partial f_1}{\partial z} z \right)  + 2 B y f_2, \hfill \vspace{6pt} \\
\dot{Y} & = 1, \hfill \vspace{6pt} \\
\dfrac{\varepsilon}{\alpha^2} \dot{Z} & =  \dfrac{\varepsilon}{\alpha^2}  C f_2 + D\left( \dfrac{\partial g_1}{\partial x} x + \dfrac{1}{2} \dfrac{\partial^2 g_1}{\partial y^2} y^2 + \dfrac{\partial^2 g_1}{\partial y \partial z} y z + \dfrac{1}{2} \dfrac{\partial^2 g_1}{\partial z^2} z^2 \right), \hfill
\end{aligned}
\end{equation}

\smallskip

Since $ \varepsilon / \alpha^2 \ll 1$, the first term of the right-hand-side of the third equation of (A-6) can be neglected. Then, replacing in (A-6) $x$, $y$ and $z$ by the right-hand-side of (A-4) and identifying with the following system in which we have posed: $( \varepsilon / \alpha^2) = \epsilon$:

\begin{equation}
\label{eqA7}
\begin{aligned}
\dot{X} & = a Y + b Z + O \left( X, \varepsilon, Y^2, Y Z, Z^2 \right), \hfill \vspace{6pt} \\
\dot{Y} & = 1 + O \left( X, Y, Z, \varepsilon \right), \hfill \vspace{6pt} \\
\epsilon \dot{Z} & =  -\left( X + Z^2 \right) + O \left( \varepsilon X, \varepsilon Y, \varepsilon Z, \varepsilon^2, X^2 Z, Z^3, X Y Z \right), \hfill
\end{aligned}
\end{equation}

\smallskip
we find:

\begin{equation}
\label{eqA8}
\begin{aligned}
a & = A \left( \dfrac{\partial f_1}{\partial x_2} - \dfrac{C}{D} \dfrac{\partial f_1}{\partial y_1} \right)f_2 + 2 Bf_2^2 , \hfill \\
b & = \dfrac{A}{D} \dfrac{\partial f_1}{\partial y_1}, \hfill  \\
\end{aligned}
\end{equation}

\smallskip

where

\begin{equation}
\label{eqA9}
\begin{aligned}
A & = \frac{1}{2} \dfrac{\partial g_1}{\partial x} \dfrac{\partial^2 g_1}{\partial z^2}, \hfill \vspace{6pt} \\
B & = \dfrac{1}{4} \left[\dfrac{\partial^2 g_1}{\partial y^2}\dfrac{\partial^2 g_1}{\partial z^2} - \left(\dfrac{\partial^2 g_1}{\partial y \partial z}\right)^2  \right], \hfill \vspace{6pt} \\
C & = - \frac{1}{2} \dfrac{\partial^2 g_1}{\partial y \partial z}, \hfill \vspace{6pt} \\
D & = - \frac{1}{2} \dfrac{\partial^2 g_1}{\partial z^2}. \hfill
\end{aligned}
\end{equation}

\smallskip

Finally, we deduce:

\begin{equation}
\label{eqA10}
\begin{aligned}
a = & \frac{1}{2} f_2^2 \left( \dfrac{\partial^2 g_1}{\partial x_2^2 } \dfrac{\partial^2 g_1}{\partial y_1^2} - (\dfrac{\partial^2 g_1}{\partial x_2 \partial y_1})^2 \right) + \frac{1}{2} f_2 \dfrac{\partial g_1}{\partial x_1} \left( \dfrac{\partial^2 g_1}{\partial y_1^2 }\dfrac{\partial f_1}{\partial x_2} - \dfrac{\partial^2 g_1}{\partial x_2 \partial y_1 } \dfrac{\partial f_1}{\partial y_1} \right), \hfill \\
b = & - \dfrac{\partial g_1}{\partial x_1}\dfrac{\partial f_1}{\partial y_1}, \hfill  \\
\end{aligned}
\end{equation}

\smallskip

This is the result established by Beno\^{i}t \cite{Benoit1990} and presented in Sec. 3.7.

\subsection*{B. Normal form of 4D singularly perturbed systems with one fast variable}

Let's consider the four-dimensional \textit{singularly perturbed dynamical system} (\ref{eq30}) with $k=3$ \textit{slow} variables and $m=1$ \textit{fast} and let's make the following change of variables:

\begin{equation}
x_1 = \alpha^2 x \mbox{, }  x_2 = \alpha y \mbox{, }  x_3 = \alpha z \mbox{, } y_1 = \alpha u \mbox{ where } \alpha \ll 1.
\end{equation}

By taking into account extension of Beno\^{i}t's generic hypothesis Eqs. (\ref{eq40},\ref{eq41}) and while using Taylor series expansion the system (\ref{eq30}) becomes:

\begin{equation}
\label{eqA12}
\begin{aligned}
\dot{x} = & \dfrac{\partial f_1}{\partial y} y + \dfrac{\partial f_1}{\partial z} z + \dfrac{\partial f_1}{\partial u} u, \hfill \vspace{6pt} \\
\dot{y} = & f_2 \left( x, y, z, u \right), \hfill \vspace{6pt} \\
\dot{z} = & f_3 \left( x, y, z, u \right), \hfill \vspace{6pt} \\
\dfrac{\varepsilon}{\alpha^2} \dot{u} & = \dfrac{\partial g_1}{\partial x} x + \dfrac{1}{2} \dfrac{\partial^2 g_1}{\partial y^2} y^2  + \dfrac{1}{2} \dfrac{\partial^2 g_1}{\partial z^2} z^2  + \dfrac{1}{2} \dfrac{\partial^2 g_1}{\partial u^2} u^2 + \dfrac{\partial^2 g_1}{\partial y \partial z} y z +  \dfrac{\partial^2 g_1}{\partial y \partial u} y u + \dfrac{\partial^2 g_1}{\partial z \partial u} z u . \hfill
\end{aligned}
\end{equation}

\newpage

Then, let's make the standard polynomial change of variables:

\begin{equation}
\label{eqA13}
\begin{aligned}
X &  = A x + B y^2 + C z^2, \hfill \vspace{6pt} \\
Y & = \dfrac{y}{f_2}, \hfill \vspace{6pt} \\
Z & = \dfrac{z}{f_3} + D y, \hfill \vspace{6pt} \\
U & = E y + F z + G u. \hfill
\end{aligned}
\end{equation}

\smallskip

From (A-13) we deduce that:

\begin{equation}
\label{eqA14}
\begin{aligned}
x & = \frac{X - B f_2^2 Y^2 - C f_3^2 \left( Z - D f_2 Y \right)^2}{A}, \hfill \vspace{6pt} \\
y & = f_2 y, \hfill \vspace{6pt} \\
z & = f_3 \left( Z - D f_2 Y \right), \hfill \vspace{6pt} \\
u & = \frac{U - E f_2 Y - F f_3 \left( Z - D f_2 Y \right) }{G}. \hfill
\end{aligned}
\end{equation}

\smallskip

The time derivative of system (A-13) gives:

\begin{equation}
\label{eqA15}
\begin{aligned}
\dot{X} & = A \dot{x} + 2 B y \dot{y} + 2 C z \dot{z}, \hfill \vspace{6pt} \\
\dot{Y} & = \dfrac{\dot{y}}{f_2}, \hfill \vspace{6pt} \\
\dot{Z} & = \dfrac{\dot{z}}{f_3} + D \dot{y}, \hfill \vspace{6pt} \\
\dot{U} & =  E\dot{y} + F\dot{z} + G \dot{u}. \hfill
\end{aligned}
\end{equation}

\smallskip

Then, multiplying the fourth equation of (A-15) by $( \varepsilon / \alpha^2)$ and while replacing in (A-15) $\dot{x}$, $\dot{y}$, $\dot{z}$ and $\dot{u}$ by the right-hand-side of system (A-12) leads to:

\begin{equation}
\label{eqA16}
\begin{aligned}
\dot{X} = & A \left( \dfrac{\partial f_1}{\partial y} y + \dfrac{\partial f_1}{\partial z} z  + \dfrac{\partial f_1}{\partial u} u\right)  + 2 B y f_2 + 2 C z f_3, \hfill \vspace{6pt} \\
\dot{Y} = & 1, \hfill \vspace{6pt} \\
\dot{Z} = & 1 + D f_2, \hfill \vspace{6pt} \\
\dfrac{\varepsilon}{\alpha^2} \dot{U} = &  \dfrac{\varepsilon}{\alpha^2} E f_2 + \dfrac{\varepsilon}{\alpha^2} F f_3 \hfill + G\left( \dfrac{\partial g_1}{\partial x} x + \ldots + \dfrac{\partial^2 g_1}{\partial z \partial u} z u \right), \hfill
\end{aligned}
\end{equation}

\smallskip

Since $ \varepsilon / \alpha^2 << 1$, the two first terms of the right-hand-side of the fourth equation of (A-16) can be neglected. Then, by replacing in (A-16) $x$, $y$, $z$ and $u$ by the right-hand-side of (A-14) and by identifying with the following system in which we have posed: $( \varepsilon / \alpha^2) = \epsilon$:

\begin{equation}
\label{eqA17}
\begin{aligned}
\dot{X} & = \tilde{a} Y + \tilde{b} U + O \left( X, \epsilon, Y^2, Y U, U^2 \right), \hfill \vspace{6pt} \\
\dot{Y} & = 1 + O \left( X, Y, U, \epsilon \right), \hfill \vspace{6pt} \\
\dot{Z} & = 1 + O \left( X, Y, U, \epsilon \right), \hfill \vspace{6pt} \\
\epsilon \dot{Z} & =  -\left( X + U^2 \right) + O \left( \epsilon X, \epsilon Y, \epsilon U, \epsilon^2, X^2 U, U^3, X Y U  \right), \hfill
\end{aligned}
\end{equation}

\smallskip

we find:

\begin{equation}
\label{eqA18}
\begin{aligned}
\tilde{a} = & A \left( \dfrac{\partial f_1}{\partial x_2} - \dfrac{E}{G} \dfrac{\partial f_1}{\partial y_1} \right)f_2 + A \left( \dfrac{\partial f_1}{\partial x_3} - \dfrac{F}{G} \dfrac{\partial f_1}{\partial y_1} \right) + 2 B f_2^2 + 2 C f_3^2, \hfill \\
\tilde{b} = & \dfrac{A}{G} \dfrac{\partial f_1}{\partial y_1}, \hfill  \\
\end{aligned}
\end{equation}

\smallskip
where

\begin{equation}
\label{eqA19}
\begin{aligned}
A & = \frac{1}{2} \dfrac{\partial g_1}{\partial x} \dfrac{\partial^2 g_1}{\partial u^2}, \hfill \vspace{6pt} \\
B & = \dfrac{f_3}{2f_2} \left[ \dfrac{\partial^2 g_1}{\partial u^2}\dfrac{\partial^2 g_1}{\partial y \partial z} + \dfrac{\partial^2 g_1}{\partial y \partial u}\dfrac{\partial^2 g_1}{\partial z \partial u} \right] + \dfrac{1}{4} \left[ \dfrac{\partial^2 g_1}{\partial u^2}\dfrac{\partial^2 g_1}{\partial y^2}  - \left( \dfrac{\partial^2 g_1}{\partial y \partial u} \right)^2 \right], \hfill \vspace{6pt} \\
C & = \dfrac{1}{4} \left[\dfrac{\partial^2 g_1}{\partial z^2}\dfrac{\partial^2 g_1}{\partial u^2} - \left(\dfrac{\partial^2 g_1}{\partial z \partial u}\right)^2  \right], \hfill \vspace{6pt} \\
D & = - \frac{1}{f_2}, \hfill \vspace{6pt} \\
E & = - \frac{1}{2} \dfrac{\partial^2 g_1}{\partial y \partial u}, \hfill \vspace{6pt} \\
F & = - \frac{1}{2} \dfrac{\partial^2 g_1}{\partial z \partial u}, \hfill \vspace{6pt} \\
G & = - \frac{1}{2} \dfrac{\partial^2 g_1}{\partial u^2}. \hfill
\end{aligned}
\end{equation}

Finally, we deduce:

\begin{equation}
\label{eqA20}
\begin{aligned}
\tilde{a} & = \frac{1}{2} f_2^2 \left( \dfrac{\partial^2 g_1}{\partial x_2^2 } \dfrac{\partial^2 g_1}{\partial y_1^2} - (\dfrac{\partial^2 g_1}{\partial x_2 \partial y_1})^2 \right) + \frac{1}{2} f_2 \dfrac{\partial g_1}{\partial x_1} \left( \dfrac{\partial^2 g_1}{\partial y_1^2 }\dfrac{\partial f_1}{\partial x_2} - \dfrac{\partial^2 g_1}{\partial x_2 \partial y_1 } \dfrac{\partial f_1}{\partial y_1} \right) \hfill \\
& + \frac{1}{2} f_3^2 \left( \dfrac{\partial^2 g_1}{\partial x_3^2 } \dfrac{\partial^2 g_1}{\partial y_1^2} - (\dfrac{\partial^2 g_1}{\partial x_3 \partial y_1})^2 \right) + \frac{1}{2} f_3 \dfrac{\partial g_1}{\partial x_1} \left( \dfrac{\partial^2 g_1}{\partial y_1^2 }\dfrac{\partial f_1}{\partial x_3} - \dfrac{\partial^2 g_1}{\partial x_3 \partial y_1 } \dfrac{\partial f_1}{\partial y_1} \right) \hfill \\
& + f_2 f_3 \left( \dfrac{\partial^2 g_1}{\partial x_2 \partial x_3 } \dfrac{\partial^2 g_1}{\partial y_1^2} - \dfrac{\partial^2 g_1}{\partial x_2 \partial y_1}\dfrac{\partial^2 g_1}{\partial x_3 \partial y_1} \right), \hfill \\
\tilde{b} & = - \dfrac{\partial g_1}{\partial x_1}\dfrac{\partial f_1}{\partial y_1}, \hfill  \\
\end{aligned}
\end{equation}

\smallskip

This is the result we established in Sec. 4.7. Moreover, let's notice that by posing $f_3=0$ in $\tilde{a}$ we find again $a$ given in Sec. 3.7.

Routh-Hurwitz' theorem and their application to the determination of the Hopf bifurcation parameter-value in the case of three and four-dimensional singularly perturbed system are presented in this appendix.

\subsection*{C. Routh-Hurwitz's theorem for 3D systems}

According to (\ref{eq23}) the Cayley-Hamilton eigenpolynomial associated with the Jacobian of a three-dimensional singularly perturbed system (\ref{eq11}) reads:

\begin{equation}
\label{eqA21}
\lambda^3 - \sigma_1 \lambda^2 + \sigma_2 \lambda - \sigma_3 = 0
\end{equation}

where

\begin{equation}
\label{eqA22}
\begin{aligned}
\sigma_1 & = \lambda_1 + \lambda_2 + \lambda_3,\\
\sigma_2 & = \lambda_1\lambda_2 + \lambda_2\lambda_3 + \lambda_1\lambda_3,\\
\sigma_3 & = \lambda_1\lambda_2\lambda_3.
\end{aligned}
\end{equation}

Let's rewrite the eigenpolynomial (A-21) as: $a_3 \lambda^3 + a_2 \lambda^2 + a_1 \lambda + a_0 = 0$ ($a_0 > 0$). Routh-Hurwitz' theorem \cite{Routh1877,Hurwotz1893} states that the real parts of the eigenvalues of this eigenpolynomial are negative \textit{if and only if} all the following determinants:

\begin{equation}
\label{eqA23}
D_1 = a_1 \quad \mbox{ ; } \quad D_2 = \begin{vmatrix} a_1  &  a_0  \\ a_3  &  a_2 \end{vmatrix} = a_1 a_2 - a_0 a_3
\end{equation}

are positive.\\

Now, let suppose that the eigenpolynomial (A-21) has one real eigenvalue $\lambda_1 \neq 0$ and two complex conjugated $\lambda_{2,3} = a + \imath b$ (with $a\neq 0$ an $b \neq 0$). So, we have:

\begin{equation}
\label{eqA24}
\begin{aligned}
\sigma_1 & = \lambda_1 + 2a,\\
\sigma_2 & = 2a\lambda_1 + a^2 + b^2,\\
\sigma_3 & = \lambda_1\left( a^2 + b^2 \right).
\end{aligned}
\end{equation}

The determinant $D_2$ reads:

\begin{equation}
\label{eqA25}
D_2 = -2 a \left(a^2 + b^2 + 2a \lambda _1 + \lambda _1^2 \right)
\end{equation}

Moreover, if we consider that the real part of the complex conjugated eigenvalues $\lambda_{2,3}$ depends on a parameter, say $\mu$, we have $a = a\left( \mu \right)$. Then, determinant $D_2$ vanishes at the location of the points where the real part $a = a\left( \mu \right)$. So, it can be used to determine the Hopf-parameter value.

\subsection*{D. Routh-Hurwitz's theorem for 4D systems}

According to (\ref{eq43}) the Cayley-Hamilton eigenpolynomial associated with the Jacobian of a four-dimensional singularly perturbed system (\ref{eq30}) reads:

\begin{equation}
\label{eqA26}
\lambda^4 - \sigma_1 \lambda^3 + \sigma_2 \lambda^2 - \sigma_3 \lambda + \sigma_4 = 0
\end{equation}

where

\begin{equation}
\label{eqA27}
\begin{aligned}
\sigma_1 & = \lambda_1 + \lambda_2 + \lambda_3 + \lambda_4,\\
\sigma_2 & = \lambda_1 \lambda _2 + \lambda_1 \lambda_3 + \lambda_2 \lambda_3 + \lambda_1 \lambda_4 + \lambda_2 \lambda_4 + \lambda_3 \lambda_4,\\
\sigma_3 & = \lambda_1 \lambda_2 \lambda_3 + \lambda_1 \lambda_2 \lambda_4 + \lambda_1 \lambda_3 \lambda_4 + \lambda_2 \lambda_3 \lambda_4,\\
\sigma_4 & = \lambda_1\lambda_2\lambda_3\lambda_4.\\
\end{aligned}
\end{equation}

Let's rewrite the eigenpolynomial (A-26) as: $a_4 \lambda^4 + a_3 \lambda^3 + a_2 \lambda^2 + a_1 \lambda + a_0 = 0$ ($a_0 > 0$). Routh-Hurwitz' theorem [1877, 1893] states that the real parts of the eigenvalues of this eigenpolynomial are negative \textit{if and only if} all the following determinants:

\begin{equation}
\label{eqA28}
D_1 = a_1 \mbox{ ; } D_2 = \begin{vmatrix} a_1  &  a_0  \\ a_3  &  a_2  \end{vmatrix} = a_1 a_2 - a_0 a_3 \mbox{ ; } \quad D_3 = \begin{vmatrix} a_1 & a_0  &  0  \\ a_3  &  a_2  &  a_1  \\ 0  &  a_4  &  a_3  \end{vmatrix}
\end{equation}

are positive.\\

Now, let suppose that the eigenpolynomial (A-26) has two real eigenvalues $\lambda_1$, $\lambda_2$ with $\lambda_1 \neq - \lambda_2 \neq 0$ and two complex conjugated $\lambda_{3,4} = a + \imath b$ (with $a\neq 0$ an $b \neq 0$). So, we have:

\begin{equation}
\label{eqA29}
\begin{aligned}
\sigma_1 & = 2 a + \lambda_1 + \lambda_2,\\
\sigma_2 & = a^2 + b^2 + 2a \left( \lambda_1 + \lambda_2 \right) + \lambda_1 \lambda_2,\\
\sigma_3 & =  2 a \lambda_1 \lambda_2 + \left(a^2 + b^2\right) \left( \lambda_1 + \lambda_2 \right),\\
\sigma_4 & = \left(a^2 + b^2 \right) \lambda_1 \lambda_2.
\end{aligned}
\end{equation}

The determinant $D_3$ reads:

\begin{equation}
\label{eqA30}
D_3 = 2 a \left(a^2 + b^2 + 2a \lambda_1 + \lambda_1^2 \right) \left( \lambda_1 + \lambda_2 \right) \left(a^2 + b^2 + 2a\lambda_2 + \lambda_2^2 \right)
\end{equation}

Moreover, if we consider that the real part of the complex conjugated eigenvalues $\lambda_{2,3}$ depends on a parameter, say $\mu$, we have $a = a\left( \mu \right)$. Then, determinant $D_3$ vanishes at the location of the points where the real part $a = a\left( \mu \right)$. So, it can be used to determine the Hopf-parameter value.

\end{article}

\begin{references}

\bibitem{Argemi} J. Arg\'{e}mi, J. [1978] \emph{Approche qualitative d'un probl\`{e}me de perturbations singuli\`{e}res dans
$\mathbb{R}^4$}, in Equadiff 1978, ed. R. Conti, G. Sestini, G. Villari (1978), 330--340.

\bibitem{BenCalDien} E. Beno\^{i}t, J.L. Callot, F., Diener and M. Diener, \emph{Chasse au canard}, Collectanea Mathematica (31--32) (1-3) (1981),
37--119.

\bibitem{Benoit1981a} E. Beno\^{i}t, \emph{Tunnels et entonnoirs}, CR. Acad. Sc. Paris \textbf{292}, S\'{e}rie I (1981) 283--286.

\bibitem{Benoit1981b} E. Beno\^{i}t, \emph{\'{E}quations diff\'{e}rentielles : relation entr\'{e}e-sortie}, CR. Acad. Sc. Paris \textbf{293}, S\'{e}rie I (1981) 293--296.

\bibitem{BenoitLobry} E. Beno\^{i}t and C. Lobry, \emph{Les canards de $\mathbb{R}^3$}, CR. Acad. Sc. Paris \textbf{294}, S\'{e}rie I (1982) 483--488.

\bibitem{Benoit1983} E. Beno\^{i}t, \emph{Syst\`{e}mes lents-rapides dans $\mathbb{R}^3$ et leurs canards}, Soci\'{e}t\'{e} Math\'{e}matique de France,  Ast\'{e}risque, (190--110) (1983) 159--191.

\bibitem{Benoit1984} E. Beno\^{i}t, Canards de $\mathbb{R}^3$, Th\`{e}se d'\'{e}tat (PhD), Universit\'{e} de Nice, 1984.

\bibitem{Benoit1990} E. Beno\^{i}t, \emph{Canards et enlacements}, Publications de l'Institut des Hautes Etudes Scientifiques, \textbf{72} (1990) 63--91.

\bibitem{Benoit2001} E. Beno\^{i}t, \emph{Perturbation singuli\`{e}re en dimension trois~: Canards en un point pseudo singulier noeud}, Bulletin de la Soci\'{e}t\'{e} Math\'{e}matique de France, (129-1) (2001) 91--113.

\bibitem{CallotDiener1978} J.L. Callot, F. Diener and M. Diener, \emph{Le probl\`{e}me de la ``chasse au canard''}, CR. Acad. Sc. Paris, \textbf{286}, S\'{e}rie A (1978) 1059--1061.

\bibitem{Chua1971} L.O. Chua, \emph{Memristor -- The Missing Circuit Element}, IEEE Transactions on Circuit Theory, {\bf 18} (5) (1971) 507--519.

\bibitem{DiVentraPershinChua2009} M. Di Ventra, Y.V. Pershin and L.O. Chua, \emph{Circuit elements with memory: memristors, memcapacitors and meminductors}, Proceedings of the IEEE, \textbf{97} (2009) 1717--1724.

\bibitem{Diener1984} M. Diener, \emph{The Canard Unchained or How Fast/Slow Dynamical Systems Bifurcate}, Math. Intellingencer, \textbf{6}(3) (1984) 38--49.

\bibitem{Fen1971} N. Fenichel, \emph{Persistence and smoothness of invariant manifolds for flows}, Ind. Univ. Math. J., \textbf{21} (1971) 193--225.

\bibitem{Fen1974} N. Fenichel, \emph{Asymptotic stability with rate conditions}, Ind. Univ. Math. J., \textbf{23} (1974) 1109--1137.

\bibitem{Fen1977} N. Fenichel, \emph{Asymptotic stability with rate conditions II}, Ind. Univ. Math. J., \textbf{26} (1977) 81--93.

\bibitem{Fen1979} N. Fenichel, \emph{Geometric singular perturbation theory for ordinary differential equations}, J. Diff. Eq. (1979) 53--98.

\bibitem{Fitch2012} A. Fitch, D. Yu, H. Iu and V. Sreeram, \emph{Hyperchaos In A Memristor-Based Modified Canonical Chua's Circuit}, Int. J. of Bifurcation and Chaos, \textbf{22} (6) (2012) 1250133.\\

\bibitem{Fitch2013} A. Fitch and H. Iu, Development of Memristor Based Circuits, \emph{World Scientific Series on Nonlinear Science, Series A} {\bf 82} (World Scientific, Singapore), 2013.

\bibitem{Fruchard2007} A. Fruchard and R. Sch\"{a}fke, \emph{Sur le retard \`{a} la bifurcation}, In T.~Sari, editor, Colloque de Saint Louis (S\'{e}n\'{e}gal). ARIMA, \textbf{9} (2007) 431--468.

\bibitem{GinouxLLibre2011} J.M. Ginoux and J. Llibre, \emph{Flow curvature method applied to canard explosion}, Journal of Physics A: Mathematical and Theoretical, \textbf{44} (46) (2011) 465203.

\bibitem{GinouxLLibre2013} J.M. Ginoux, J. Llibre and L.O. Chua, \emph{Canards from Chua's circuit}, Int. J. of Bifurcation and Chaos, \textbf{23} (4) (2013) 1330010.

\bibitem{GinouxRossetto2013} J.M. Ginoux and B. Rossetto, The Singing Arc: The Oldest Memristor? in \textit{Chaos, CNN, Memristors and Beyond: A Festschrift for Leon Chua}, World Scientific Publishing, A. Adamatsky and G. Chen (Eds).

\bibitem{GuckenHaiduc2005} J. Guckenheimer and R. Haiduc, \emph{Canards at folded nodes}, Mosc. Math.~J., \textbf{5}(1) (2005) 91--103.

\bibitem{Hurwitz1893} A. Hurwitz, \emph{\"{U}ber die Bedingungen, unter welchen eine Gleichung nur Wurzeln mit negativen reellen Theilen besitzt}, Math. Ann., \textbf{41} (1893) 403--442.

\bibitem{ItohChua1992} M. Itoh and L.O. Chua, \emph{Canards and chaos in nonlinear systems}, Circuits and Systems, 1992. ISCAS'92. Proceedings, \textbf{6} (1992) 2789--2792.

\bibitem{ItohChua2008} M. Itoh and L.O. Chua, \emph{Memristors oscillators}, Int. J. of Bifurcation and Chaos, \textbf{18} (11) (2008) 3183--3206.

\bibitem{ItohChua2013} M. Itoh and L.O. Chua, \emph{Duality of Memristors}, Int. J. of Bifurcation and Chaos, \textbf{23} (1) (2013) 1330001.

\bibitem{Jones1994} C.K.R.T. Jones, \emph{Geometric Singular Perturbation Theory in Dynamical Systems}, \textit{Montecatini Terme}, L. Arnold, Lecture Notes in Mathematics, vol. 1609, Springer-Verlag (1994) 44--118.

\bibitem{Kaper1999} T. Kaper, \emph{An Introduction to Geometric Methods and Dynamical Systems Theory for Singular Perturbation Problems}, in Analyzing multiscale phenomena using singular perturbation methods, Baltimore, MD, (1998) 85--131. Amer. Math. Soc., Providence, RI.

\bibitem{MuthuswamyKokate2009} B. Muthuswamy and P.P. Kokate, \emph{Memristorbased chaotic circuits}, IETE Tech. Rev., \textbf{26} (2009) 417–-429.

\bibitem{Muthuswamy2010} B. Muthuswamy, \emph{Implementing memristor based chaotic circuits}, Int. J. of Bifurcation and Chaos, \textbf{20} (2010) 1335-–1350.

\bibitem{MuthuswamyChua2010} B. Muthuswamy and L.O. Chua, \emph{Simplest chaotic circuit}, Int. J. of Bifurcation and Chaos, \textbf{20} (2010) 1567-–1580.

\bibitem{Nelson1977} E. Nelson, \emph{Internal Set Theory: a new approach to nonstandard analysis}, Bull. Amer. Math. Soc., \textbf{83}(6) (1977) 1165--1198.

\bibitem{OMalley1974} R.E. O'Malley, \textit{Introduction to Singular Perturbations}, Academic Press, New York, 1974.

\bibitem{PershinDiVentra2009} Y.V. Pershin and M. Di Ventra, \emph{Experimental demonstration of associative memory with memristive neural networks}, 2009 available: http://arXiv.org/abs/arXiv:0905.2935.

\bibitem{Pontryagin1957} L.S. Pontryagin, \emph{The asymptotic behaviour of systems of differential equations with a small parameter multiplying the highest derivatives}, Izv. Akad. Nauk. SSSR, Ser. Mat., \textbf{21}(5) (1957) 605--626.

\bibitem{Robinson1966} A. Robinson, \textit{Nonstandard Analysis}, North-Holland, Amsterdam, 1966.

\bibitem{Routh1877} E.J. Routh, \textit{A Treatise on the Stability of a Given State of Motion: Particularly Steady Motion}, Macmillan and co, 1877.


\bibitem{Strukhov2008} D.B. Strukhov, G. S. Snider, G. R. Stewart and R.S. Williams, \emph{The missing memristor found}, Nature, {\bf 453} (2008) 80-–83.

\bibitem{SzmolyanWechselberger2001} P. Szmolyan and M. Wechselberger, \emph{Canards in $\mathbb{R}^3$}, J. Dif. Eqs., \textbf{177} (2001) 419--453.

\bibitem{Takens1976} F. Takens, \emph{Constrained equations, a study of implicit differential equations and their discontinuous solutions}, in \textit{Structural
stability, the theory of catastrophes and applications in the sciences}, \textit{Springer Lecture Notes in Math.}, \textbf{525} (1976) 143--234.

\bibitem{Tikhonov1948} A.N. Tikhonov, \emph{On the dependence of solutions of differential equations on a small parameter}, Mat. Sbornik N.S., \textbf{31} (1948) 575--586.

\bibitem{Tsuneda2005} A. Tsuneda, \emph{A Gallery Of Attractors From Smooth Chua's Equation}, Int. J. of Bifurcation and Chaos, \textbf{15}(1) (2005) 1--49.

\bibitem{VdP1926} B. Van der Pol, \emph{On relaxation-oscillations}, The London, Edinburgh, and Dublin Philosophical Magazine and Journal of Science, {\bf 7} (2) (1926) 978--992.

\bibitem{Wechselberger2005} M. Wechselberger, \emph{Existence and Bifurcation of Canards in $\mathbb{R}^3$ in the case of a Folded Node}, SIAM J. Applied Dynamical Systems, \textbf{4} (2005) 101--139.

\bibitem{Wechselberger2012} M. Wechselberger, \emph{\`{A} propos de canards}, Trans. Amer. Math. Soc., \textbf{364} (2012) 3289--3309.


\end{references}
\end{document}